\documentclass{article}
\usepackage[english]{babel}
\usepackage[latin1]{inputenc}
\usepackage{graphicx}
\usepackage{amssymb}
\usepackage{amsmath}
\usepackage{amsthm}
\usepackage{enumerate}
\usepackage{color}
\newtheorem{theorem}{Theorem}[section]
\newtheorem{lemma}[theorem]{Lemma}
\newtheorem{corollary}[theorem]{Corollary}

\theoremstyle{definition}
\newtheorem{example}[theorem]{Examples}

\newtheorem{definition}[theorem]{Definition}

\theoremstyle{remark}
\newtheorem{remark}[theorem]{Remark}
\numberwithin{equation}{section}

\newcommand{\rn}{\ensuremath{\mathbb R}^{n}}
\newcommand{\nat}{\ensuremath{\mathbb N}}
\newcommand{\no}{\ensuremath{{\mathbb N}_0}}

\newcommand{\ganz}{\ensuremath{\mathbb Z}}
\newcommand{\real}{\ensuremath{\mathbb R}} 
\newcommand{\beq}{\begin{equation}}
\newcommand{\eeq}{\end{equation}}
\newcommand{\bli}{\begin{list}{}{\labelwidth1.7em\leftmargin2.1em}}
\newcommand{\eli}{\end{list}}
\newcommand{\bit}{\begin{itemize}}
\newcommand{\eit}{\end{itemize}}
\newcommand{\dint}{\ensuremath{\,\mathrm{d}}}

\newcommand{\id}{\ensuremath\mathrm{id}}
\newcommand{\ve}{\ensuremath\varepsilon}
\newcommand{\wt}[1]{\ensuremath\widetilde{#1}}

\begin{document}

\title{Optimal Calder\'on Spaces for generalized Bessel potentials}
\date{\today}

\author{Elza Bakhtigareeva, Mikhail L. Goldman, and Dorothee D. Haroske}

\maketitle

\begin{abstract}

 In the paper we investigate the properties of spaces with generalized smoothness, such as Calder\'on spaces that include the classical Nikolskii-Besov spaces and many of their generalizations, and describe differential properties of generalized Bessel potentials that include classical Bessel potentials and Sobolev spaces. Kernels of potentials may have non-power singularity at the origin. With the help of order-sharp estimates for moduli of continuity of potentials, we establish the criteria of embeddings of potentials into Calder\'on spaces, and describe the optimal spaces for such embeddings.

\end{abstract}

\section{Introduction}
\label{intro}

The paper is devoted to generalized Bessel potentials constructed by the convolutions of generalized Bessel-McDonald kernels with functions from the basic rearrangement invariant space. If the criterion is satisfied for the embedding of potentials into the space of bounded continuous functions, we state the equivalent description for the cones of moduli of continuity of potentials in the uniform norm. This gives the opportunity to obtain the criterion for the embedding of potentials into the Calder\'on space. We develop here the results of \cite{GoHa}. Some results presented here we announced in papers \cite{GoHa-3,GoHa-4}. In the case of generalized Bessel potentials constructed over the basic weighted Lorentz space we describe explicitly the optimal Calder\'on space for such an embedding. The results of Sections 3, 4 are based on an application of some results obtained in \cite{BaGo}.

The paper is organized as follows. In Section~\ref{sect-1} notation, essential concepts and definitions are presented. We present the results concerning equivalent descriptions for the cone of moduli of continuity for generalized Bessel potentials in uniform norm (Theorem \ref{theo-g-1.1}), and prove two-sided estimates for a variant of the continuity envelope function for the space of potentials (Theorem \ref{theo-g-ad-2.11}). The criterion of embedding for the space of potentials into the Calder\'on space is presented in Theorem \ref{theo-g-2.4}. Sections~\ref{sect-3}-\ref{sect-3b} are devoted to the description of the optimal Calder\'on space for such embedding. Theorem \ref{theo-g-3.3} gives an explicit description in the case when the basic space for the potential is a weighted Lorentz space. The proofs of the main results of Section~\ref{sect-3} we give in Sections~\ref{sect-3a} and~\ref{sect-3b}. Section~\ref{sect-4} contains some explicit descriptions of the optimal Calder\'on space.

\section{Preparations}
\label{sect-1}

First we fix some notation. By $\nat$ we denote the set of natural numbers, 
by $\nat_0$ the set $\nat \cup \{ 0\}$. 
For two positive real sequences 
$\{\alpha_k\}_{k\in \nat}$ and $\{\beta_k\}_{k\in \nat}$ we mean by 
$\alpha_k\sim \beta_k$ that there exist constants $c_1,c_2>0$ such that $c_1 \alpha_k\leq
\beta_k\leq c_2  \alpha_k$ for all $k\in \nat$; similarly for positive
functions. Usually $B_r =  \{y\in \rn: |y|<r\}$ stands for a ball in $\rn$ centred at the origin and with radius $r>0$. 
We denote by $\mu_n$ the Lebesgue measure on $\real_n$, $n\in \nat$, and by  $V_{n}$ the volume of the  $n$-dimensional unit ball, that is, $V_n = \mu_n(B_1)$. For a set $A$ we denote by $\chi_A$ its characteristic function.


Given two (quasi-) Banach spaces $X$ and $Y$, we write $X\hookrightarrow Y$
if $X\subset Y$ and the natural embedding of $X$ in $Y$ is continuous.

All unimportant positive constants will be denoted by $c$, occasionally with
subscripts.

\subsection{Banach function spaces}
\label{sec-1-1}
We use some notation and general facts from the theory of Banach function
spaces and rearrangement-invariant spaces; for general background material we refer to \cite{BS}. As usual 
we call $f^\ast$ the decreasing rearrangement of the function $f$, i.e., $0\leq f^\ast$ is a decreasing, right-continuous function on $\real_+ = (0,\infty)$, equi-measurable with $f$, 
\beq
\mu_n(\{x\in\rn: |f(x)|>s\}) = \mu_1(\{t\in\real_+: f^\ast(t)>s\}), \quad s>0.
\eeq

\begin{definition}\label{bfs-ris}
\bli
\item[{\bfseries\upshape (i)}]
A linear space $X=X(0,T)$ of measurable functions on $(0,T)$, equipped with the norm $\|\cdot\|_X$, is called a {\em  Banach function space}, shortly: {\em BFS}, if the following conditions are satisfied:
\bit
\item[{\bfseries\upshape (P1)}]
$\|f\|_X=0 \iff f=0\ $ $\mu$-a.e. on $(0,T)$;
\item[{\bfseries\upshape (P2)}]
$|f|\leq g,\  g\in X \quad\text{implies}\quad f\in X, \ \|f\|_X \leq \|g\|_X;$
\item[{\bfseries\upshape (P3)}]
If for some functions $f_n\in X$, $f_n\geq 0$, $n\in\nat$, and $f_n$ monotonically increasing to $f$, then  $\|f_n\|_X \to \|f\|_X$, i.e.
\beq
\label{def-bfs}
\|f\|_X=\lim_{n\to\infty}\|f_n\|_X.
\eeq
\item[{\bfseries\upshape (P4)}]
For every measurable set $B\subset (0,T)$ with $\mu(B)>0$, there exists some $c_B>0$, such that for all $f\in X$,
\[
\int_B |f| \dint \mu \leq c_B \|f\|_X.
\]
\item[{\bfseries\upshape (P5)}]
For every measurable set $B\subset (0,T)$ with $\mu(B)>0$, $\|\chi_B\|_X<\infty$.
\eit
\item[{\bfseries\upshape (ii)}]
A BFS $E$ is called a 
{\em rearrangement-invariant space}, shortly: {\em RIS}, if its norm is monotone with respect to rearrangements,
\beq
f^\ast \leq g^\ast, \ g\in E \quad \text{implies}\quad f\in E, \ \|f\|_E\leq \|g\|_E.\label{1.3}
\eeq
\eli
\end{definition}

\begin{remark}\label{r1}
	Note that the limit on the right-hand side of \eqref{def-bfs} always exists (finite or infinite), because the sequence of norms  increases. The finiteness of this limit is a criterion for $f \in X$.
\end{remark}

For an RIS $E(\rn)$ its associate space $E'(\rn)$ is an RIS again, equipped with the norm
\[
 \|g\|_{E'(\rn)} = \sup_{\|f\|_{E(\rn)}\leq 1} \; \int\limits_0^\infty f^\ast(s) g^\ast(s)
\dint s,\quad g\in E'(\rn),
\]
see \cite[Ch.~2]{BS} for further details.

\begin{remark}
The following Luxemburg representation formula is known: For an RIS $E(\rn)$ there exists a unique
RIS $\widetilde{E}(\real_+)$ such that
\beq
\|f\|_{E(\rn)} = \|f^\ast\|_{\widetilde{E}(\real_+)}.\label{1.4}
\eeq
Likewise, let $\widetilde{E}'(\real_+)$ be the Luxemburg representation for $E'(\rn)$, with
\[
\|h\|_{\widetilde{E}'(\real_+)} = \sup\left\{ \int_0^\infty h^\ast(t)f^\ast(t)\dint t:\quad f\in \widetilde{E}(\real_+), \ \|f\|_{\widetilde{E}(\real_+)}\leq 1\right\}.
\]

\end{remark}

\begin{example}\label{ex-Lorentz}
Let us mention some examples of RIS such as  ${L_{p}}$, classical Lorentz and Marcinkiewicz spaces ${\Lambda _{pq}}$ and  ${M_{p}}$, Orlicz spaces and other. Recall that the norms in generalized weighted Lorentz and Marcinkiewicz spaces  ${\Lambda _{q}\left(v\right)}$ and  ${M\left(v\right)}$ with a weight  ${v}$ are given by 
\begin{align}
\|f\|_{\Lambda_q(v)} & =\left(\int _0^\infty f^\ast\left(t\right)^{q}v\left(t\right) \dint t\right)^{1/q},\quad 1\le q<\infty; \label{1.5}
\\
\|f\|_{{M\left(v\right)}} &=\sup\left\{f^{\ast\ast}\left(t\right)v\left(t\right):\ t\in \real_+\right\},\label{1.6}
\end{align}
where 
\beq
f^{\ast\ast}\left(t\right)=\frac1t \int_0^t f^\ast\left(s\right) \dint s
\label{f**}
\eeq
is the well-known maximal function of $f^\ast$, that is, $f^{\ast}\le f^{\ast\ast}$, $f^{\ast\ast}$ is monotonically decreasing, whereas $t f^{\ast\ast}(t)$ is monotonically increasing.\\

For special weights we obtain a variety of Lorentz and Marcinkiewicz spaces: For instance,   $v_{pq}\left(t\right)=t^{{\frac{q}{p}-1}}$, $ 1\le p,q<\infty $, yields the classical Lorentz spaces
\[
\Lambda_{q}\left(v_{pq}\right)=\Lambda_{pq},\quad\text{in particular,}\quad \Lambda_q(v_{qq})=\Lambda_q(1)=L_q,
\]
whereas the choice   
\[
v_b\left(t\right)=\left[t^{{\gamma }} b\left(t\right)\right]^{q} t^{{-1}},\quad \gamma\in\real,
\]
and ${b\left(t\right)}$ a slowly varying function of logarithmic type, leads to the so-called Lorentz-Karamata spaces, see for example \cite{GNO,neves-01}. 
\end{example}

\subsection{Space of Bessel potentials}\label{sec-1-2}

The space of Bessel potentials is introduced by an integral representation over an RIS  $E=E\left(\rn\right)$. Here we need the notion of the generalized Bessel-McDonald kernel $G=G_\Phi$,
\beq
G(x)=\Phi(|x|), \quad x\in \rn\setminus\{0\},
\label{g1.2}
\eeq
where we make the following assumptions on the function $\Phi:\real_+\to [0,\infty)$ everywhere in the sequel: $\Phi$ is continuous, monotonically decreasing, and
  \beq
  0<\int_0^\infty \Phi(z) z^{n-1}\dint z<\infty.
  \label{g1.3}
  \eeq
  We denote by
  \beq
  \varphi(\tau)=\Phi\left((\tau/V_n)^{1/n}\right),\quad \tau >0,
\label{g1.4}
  \eeq
   so that  $\varphi$   is a positive continuous decreasing function such that
 \begin{equation*}
 0<\int_0^\infty \varphi(\tau)\dint \tau<\infty.
 \end{equation*}

\begin{definition}\label{defi-HGE}
Let $E=E(\rn)$ be an RIS, $G$ the generalized Bessel-McDonald kernel as above. Then 
\beq
H_{E}^{G}\left(\rn\right)\equiv H_{E}^G=\left\{u=G\ast f :\ f\in E\left(\rn\right) \right\},
\label{g1.1}
\eeq
equipped with the norm 
\begin{equation*}
\|u\|_{H_{E}^G}:=\inf \left\lbrace \|f\|_{E} :\quad f \in E(\rn);\quad G\ast f=u\right\rbrace .
\end{equation*}

Here
\[
u(x)=(G\ast f)\left(x\right)=\int _{\rn} G\left(x-y\right)f(y) \dint y.
\]
\end{definition}

Let $C(\rn)$ be the space of all complex-valued bounded uniformly continuous functions on $\rn$, equipped with the sup-norm as usual.  The following estimate holds for $u=G\ast f$, with fixed $T\in\real_+$:

\begin{equation*}
\|u\|_{C}:=\sup\limits_{x \in \rn} |u(x)|\leq c_0 \int _{0}^T \varphi(\tau)f^{\ast}(\tau)\dint \tau, \quad c_{0}=1+\left( \int _{T}^{\infty} \varphi\dint \tau\right) \left( \int _{0}^T \varphi\dint \tau\right)^{-1}, 
\end{equation*}
see \cite[(1.3)]{GoMa}. Thus, if we require that
\beq
\varphi \in\widetilde{E}'(0,T),
\label{g1.5}
\eeq
then for every  $f \in E(\rn)$ such that $G\ast f=u$, we have the inequality 
\beq
\|u\|_{C}\leq c_0\|\varphi\|_{\widetilde{E}'(0,T)}\|f^{\ast}\|_{\widetilde{E}(\real_+)}=c_0\|\varphi\|_{\widetilde{E}'(0,T)}\|f\|_{{E}(\rn)}.
\label{g1.6}
\eeq
We take the infimum over such functions $f$  and obtain
\beq
\|u\|_{C}\leq c_0\|\varphi\|_{\widetilde{E}'(0,T)}\|u\|_{H_E^G(\rn)}.
\label{g1.7}
\eeq

 Under the additional conditions \eqref{g1.15},  \eqref{g1.16} we have the embedding $ H^G_E(\rn)\hookrightarrow C(\rn)$ into the space of continuous uniformly bounded functions (see Remark \ref{rem-g-1.2} below). 

For later use, let us further recall the definition of differences of functions. If $f$
is an arbitrary function on $\rn$, $h\in\rn$ and $k\in\nat$, then
\beq
(\Delta_h^k f)(x):=\sum_{j=0}^{k}\,\binom{k}{j}\,(-1)^{k-j}\,
f(x+jh), \quad x\in\rn.
\label{diff}
\eeq
Note that $\Delta_h^k$ can also be defined iteratively via
\[
(\Delta_h^1 f)(x)=f(x+h)-f(x) \quad \text{and} \quad
(\Delta_h^{k+1} f)(x)=\Delta_h^1(\Delta_h^k f)(x), \quad k\in\nat.
\]
For convenience we may write $\Delta_h$ instead of
$\Delta_h^1$. Accordingly, the $k$-th modulus of smoothness of a
function $f\in C(\rn)$ is defined by
\begin{equation}
\omega_k(f;t)=\sup_{|h|\leq t} \left\|\Delta_h^k f\right\|_{C(\rn)}, \quad
t>0, \label{modulus}
\end{equation}
 such that $\omega_1(f;t)=\omega(f;t)$. Recall that \eqref{diff} immediately gives
\beq
\omega_k(f;\lambda t) \leq \ (1+\lambda)^k \ \omega_k(f;t),\quad \lambda>0.\label{omega-dil}
\eeq

\subsection{The cone of moduli of continuity for potentials}
\label{sec-1-3}

Let $k\in\nat$, $T>0$, and $H^G_E(\rn)$ as in Definition~\ref{defi-HGE}. 
We introduce the following cone of moduli of continuity of potentials,
\beq
M=\left\{h:\real_+\to\real_+: \ h(t)= \omega_k\left(u; t^{1/n}\right), \ t\in (0,T)\quad\text{for some}\ u\in H^G_E(\rn)\right\},
\label{g1.8}
\eeq
equipped with the functional
\beq
\label{g1.9}
\varrho_M(h) = \inf\left\{ \|u\|_{H^G_E}: \ u\in H^G_E(\rn), \ \omega_k\left(u; t^{1/n}\right)= h(t),\ t\in (0,T)\right\},\quad h\in M.
\eeq
Plainly, $M=M(k,E,G,T)$, but we shall usually write $M$ for convenience. 

Our next aim is to characterize the above cone by a simpler expression. For this purpose we define the notion of covering and equivalence of cones. We claim that the cone $M=M(0,T)$  with $\varrho_M$ is covered by the cone $K=K(0,T)$ with $\varrho_K$, with covering constant $c \in (0, \infty)$, written as $M\overset{c}{\prec}K$, 
if for any function $h\in M$ there exists some function $g\in K$, such that 
\beq\label{g1.13}
\varrho_K(g)\leq c \varrho_M(h)\qquad \text{and}\qquad g(t)\geq h(t), \quad t\in (0,T).
\eeq
We write   $M\prec K$    if there exists $c \in (0, \infty)$, such that $M\overset{c}{\prec}K$.  We denote as $c_0(M\prec K)$  the best constant of covering  $M\prec K$, that is
$$c_0(M\prec K)=\inf \left\lbrace c \in \real_+: M\overset{c}{\prec}K\right\rbrace.$$
In case of mutual covering we call it {\em equivalence of cones}, that is
\beq
M\approx K\qquad \iff\qquad M\prec K\prec M.
\label{g1.14}
\eeq

We need some further notation. For $\varphi$ as above, let
\beq
\label{g1.12}
\Omega_\varphi(t,\tau)=\frac{\varphi(\tau)}{1+\left(\frac{\tau}{t}\right)^{k/n}},\quad t,\tau>0,
\eeq
and
\[
\widetilde{E}_0(0,T) = \left\{\sigma \in \widetilde{E}(0,T): \sigma\geq 0, \ \sigma\ \text{monotonically decreasing}\right\}.
\]
Then the new cone $K=K(E,\varphi,T)$ is given by
\beq
K=\left\{h:\real_+\to\real_+: \ h(t)=\int_0^T \Omega_\varphi(t,\tau) \sigma(\tau)\dint \tau\quad\text{for some}\ \sigma \in \widetilde{E}_0(0,T)\right\},
\label{g1.10}
\eeq
equipped with the functional
\beq
\label{g1.11}
\varrho_K(h)=\inf\left\{ \left\|\sigma\right\|_{\widetilde{E}(0,T)}: \sigma \in \widetilde{E}_0(0,T), \quad \int_0^T \Omega_\varphi(t,\tau) \sigma(\tau)\dint \tau = h(t)\right\}.
\eeq

Our first main result reads as follows.
\begin{theorem}
\label{theo-g-1.1}
Let $\Phi$ be a standard function with \eqref{g1.3}, $\varphi$, given by \eqref{g1.4}, satisfy \eqref{g1.5}.
\bli
\item[{\bfseries\upshape 1.}]
Assume, in addition, that $\Phi\in C^k(\real_+),$  $k\in\nat$, and $\Phi$ satisfies the following estimates for some positive real numbers $a_1, a_2, z_1$,
\begin{align} 
\max_{1\leq j\leq k} \left( z^{2j} \left|\Phi_j(z)\right|\right) \leq & \ a_1 \ \Phi(z),\quad z\in (0,z_1], \label{g1.15}\\
\max_{1\leq j\leq k} \left( z^{2j} \left|\Phi_j(z)\right|\right) \leq & \ a_2\  z^k \Phi(z),\quad z>z_1, \label{g1.16}
\end{align}
where
\[
\Phi_j(z)=\left(\frac{1}{z}\ \frac{\dint}{\dint z}\right)^j \Phi(z).
\]
Then, for the cones $M$ given by \eqref{g1.8} and $K$ given by  \eqref{g1.10} we have the covering $M \prec K$. 
\item[{\bfseries\upshape 2.}]
If $\Phi$ additionally satisfies the following estimate for some positive real number $\delta_1$: 
\begin{align} 
  (-1)^k z^k \Phi^{(k)}(z) \geq & \ \delta_1\ \Phi(z), \quad z\in (0,z_1],\label{g1.17}
\end{align}
then $M$ and $K$ are equivalent. The constants  in the condition of mutual coverings  \eqref{g1.14} depend on $k$,  $n$, $T$, $a_1,a_2$, $z_1,\delta_1$ and on the norm of the embedding operator \eqref{g1.7}.
\eli
\end{theorem}
\begin{remark}\label{g-ad2.7} The proof is based on the following crucial estimates.
\bli
\item[{\bfseries\upshape 1.}]  Under the assumptions of Theorem \ref{theo-g-1.1}, Part 1, for any  $u \in H_E^G(\rn)$,   that is $u=G\ast f,\, f \in E(\rn),$
the estimate holds
\beq 
\omega_k(u, t^{1/n})\leq c_{1}\int_{0}^{T}\Omega_{\varphi}(t, \tau)f^{*}(\tau)\dint \tau, \quad t \in (0,T), \label{g-ad2.28}
\eeq
see \cite{GoMa}, with $c_1=c_1(k, n,z_1,a_1,a_2, T,  \|\id\|)\in \real_+$, where $\|\id\|$ refers to the norm of embedding operator in \eqref{g1.7}.

\item[{\bfseries\upshape 2.}]  Under the assumptions of Theorem \ref{theo-g-1.1}, Part 2,  for  $\sigma_0 \in \widetilde{E}_0(0,T)$ we denote 
$\sigma(t) =\sigma_0(t),\, t \in (0, T),\quad \sigma(t)=0,\, t \geq T$. Then, there exists $f \in E(\rn)$, such that  $f^{*}(t)\leq \sigma(t)$, $t \in \real_+$, and for $u=G\ast f \in H_E^G(\rn)$
\beq 
\omega_k(u, t^{1/n})\geq c_{2}\int_{0}^{T}\Omega_{\varphi}(t, \tau)\sigma_0(\tau)\dint \tau,\quad t \in (0,T), \label{g-ad2.29}
\eeq
with $c_2=c_2(k, n,z_1,a_1,\delta_1, T)>0$, see \cite{GoMa-2, GoHa-3}.
Moreover, for the best constants of coverings in Theorem \ref{theo-g-1.1} we have
\beq \label{g-ad2.30}
c_0(M\prec K)\leq c_1;\quad c_0(K\prec M)\geq c_2^{-1}.
\eeq	 

\eli
\end{remark}

\begin{proof} {\em Step 1}.\quad We first show that under the assumptions \eqref{g1.15}, \eqref{g1.16} there is the covering $M \overset{c}{\prec}K$ for the cones given by \eqref{g1.8} and \eqref{g1.10} with any constant of covering  $c>c_1$. Let $c=(1+\varepsilon)^{2}c_1,\quad \varepsilon \in (0, 1).$ For $h\in M$ there exists $u_\varepsilon \in H_E^G(\rn)$  such that
  \begin{align}
    \omega_k(u_\varepsilon; t^{1/n}) & = h(t),\, t\in (0,T), \quad\|u_\varepsilon\|_{H_E^G}\leq (1+\varepsilon)\varrho_M(h).\nonumber    
  \end{align}
  Then, for $u_\varepsilon \in H_E^G(\rn)$, we find $f_\varepsilon \in E(\rn)$, such that
 $$u_\varepsilon=G\ast f_\varepsilon,\quad \|f_\varepsilon\|_E \leq (1+\varepsilon)\|u_\varepsilon\|_{H^G_E} \ \leq \ (1+\varepsilon)^2\varrho_M(h).$$
 By \eqref{g-ad2.28}  for $u=u_\varepsilon$ we have the estimate
  \begin{equation*}
  h(t)\leq  g_{\varepsilon}(t):=c_1\int_0^T \Omega_\varphi(t,\tau) f_{\varepsilon}^\ast(\tau)\dint \tau=\int_0^T \Omega_\varphi(t,\tau) \sigma_{\varepsilon, 0}(\tau)\dint \tau,\quad t\in (0,T).
  \end{equation*}
  Here
  \[
  \sigma_{\varepsilon, 0}=c_1f^{\ast}_{\varepsilon} \in \widetilde{E}_0(0,T);\quad  \left\|\sigma_{\varepsilon, 0}\right\|_{\widetilde{E}(0,T)} \leq c_1\|f^{\ast}_\varepsilon\|_{\widetilde{E}(\real_+)}= c_1\|f_\varepsilon\|_{E}.
  \]
  We see that
  $$
   h\leq  g_{\varepsilon} \in K;\quad \varrho_{K}(g_{\varepsilon})\leq \|\sigma_{\varepsilon, 0}\|_{\widetilde{E}(0,T)} \leq c_1\|f_{\varepsilon}\|_{E}\leq c_1 (1+\varepsilon)^2\varrho_M(h)=c\varrho_M(h).
 $$
  
These estimates show that
  \[
  M \overset{c}{\prec}K\quad \text{for all}\quad 
  c>c_1 \quad\text{implies}\quad 
  c_0(M\prec K)\leq c_1.
  \]

  {\em Step 2}.\quad We will show that under assumptions of Theorem \ref{theo-g-1.1}, Part 2,  there is the covering
  \beq\label{g-ad2.31} 
  K\overset{c}{\prec} M\quad \text{for all}\quad 
  c>c_2^{-1},\quad\text{which implies}\quad 
  c_0(K\prec M)\leq c_2^{-1}.
    \eeq
     Let $c=(1+\varepsilon)c_2^{-1},\, \varepsilon \in (0,1)$. For $g\in K$ there exists  $\sigma_{\varepsilon, 0} \in\widetilde{E}_0(0,T)$, such that
  \begin{align*}
   g(t)=& \int_0^T \Omega_\varphi(t,\tau)\sigma_{\varepsilon, 0}(\tau)\dint\tau,\, t \in (0,T);\quad \|\sigma_{\varepsilon, 0}\|_{\widetilde{E}(0,\infty)}\leq (1+\varepsilon)\varrho_K(g).
    \end{align*}
    Now, let                                  
    \begin{align*}    
    \sigma_{\varepsilon}(t) = \begin{cases} \sigma_{\varepsilon,0}(t), & t \in (0,T), \\ 0, & t \geq T.\end{cases}
  \end{align*}
  Then, 
  $$\|\sigma_{\varepsilon}\|_{\widetilde{E}(0,\infty)}=\|\sigma_{\varepsilon, 0}\|_{\widetilde{E}(0,T)}\leq (1+\varepsilon)\varrho_K(g).$$
 According to \eqref{g-ad2.29} with $\sigma_{0}=\sigma_{\varepsilon, 0}$  we find $f_\varepsilon \in E(\rn)$, such that $f_{\varepsilon}^{\ast}\leq \sigma_{\varepsilon}$ and for $u_{\varepsilon}=G\ast f_{\varepsilon}$  the estimate holds  $\omega_{k}(u_{\varepsilon}, t^{1/n})\geq c_{2}g(t), \, t \in (0, T)$. 
 So, we denote
 
 $$h_{\varepsilon}(t):=\omega_{k}(c_2^{-1}u_{\varepsilon}, t^{1/n})=c_2^{-1}\omega_{k}(u_{\varepsilon}, t^{1/n})\geq g(t),\quad t \in (0,T).$$
  Moreover,  $h_{\varepsilon} \in M$, and 
   \begin{align*}
  \varrho_M(h_{\varepsilon}) & \leq \left\| c_2^{-1} u_{\varepsilon}\right\|_{H^G_E} \leq \|c_2^{-1}f_{\varepsilon}\|_E = c_2^{-1}\|f_{\varepsilon}^\ast\|_{\widetilde{E}(0,\infty)}\\
 & \leq  c_2^{-1}\|\sigma_{\varepsilon}\|_{\widetilde{E}(0,\infty)} \leq (1+\varepsilon)c_2^{-1}\varrho_K(g). 
 \end{align*}
  
These estimates show that
 \[ 
 K\overset{c}{\prec}M\quad\text{for all}\quad 
 c>c_{2}^{-1} \quad\text{which implies}\quad 
 c_{0}(K\prec M)\leq c_{2}^{-1}.
 \]
\end{proof}

\begin{remark}
	\label{rem-g-1.2}
	For $\Omega_{\varphi}(t, \tau)$, see \eqref{g1.12}, $\varphi \in \widetilde{E}'(0,T)$,  see \eqref{g1.5}, and $\sigma \in \widetilde{E}_0(0,T)$ the following assertions hold
	\begin{equation*}
	\Omega_{\varphi}(t, \tau)\sigma(\tau)\leq \varphi(\tau)\sigma(\tau) \in L_1(0,T);\quad  \Omega_{\varphi}(t, \tau)\sigma(\tau)\rightarrow 0\,(t\rightarrow +0).
	\end{equation*}
	Therefore, by Lebesgue's dominated convergence theorem \eqref{g1.10} implies
	\begin{equation*}
	h \in K \Rightarrow h(t)\rightarrow 0 (t\rightarrow +0).
	\end{equation*}
	Together with the covering $M\prec K,$ see \eqref{g1.13}, it leads to 
	\begin{equation*}
	h \in M \Rightarrow h(t)\rightarrow 0 (t\rightarrow +0).
	\end{equation*}
	It means that under the assumptions \eqref{g1.15}, \eqref{g1.16}
	\begin{equation*}
	\omega_k(u,t)\rightarrow 0 (t\rightarrow +0)\Rightarrow H^G_E(\rn)\hookrightarrow C(\rn).
	\end{equation*}
\end{remark}

\begin{example}\label{exm-g-1.3}
  For the classical Bessel potentials the corresponding Bessel-McDonald kernels are determined by \eqref{g1.2} with
  \beq
  \Phi_\nu(x)=H_\nu(x),\quad x\in\real_+,\quad \nu =\frac{n-\alpha}{2},\quad 0<\alpha<n,
  \label{g1.18}
  \eeq
  where
$H_\nu(x)=x^{-\nu}  K_{\nu }(x)$, $x>0$, and $K_\nu$ is the modified Bessel function, 
\beq
K_{\nu }\left(\varrho \right)=\frac12 \left(\frac{\varrho}{2}\right)^{\nu }\int\limits_0^\infty \xi^{-\nu -1} e^{-\xi - \varrho^2/4\xi } \dint\xi
\label{1.13}
\eeq
cf. \cite{goldman-24,GoHa, nik}. From the well-known properties of these functions, it follows easily that  conditions \eqref{g1.15}--\eqref{g1.16} are satisfied and that
\beq\label{g1.19}
\Phi_\nu(y)\simeq \begin{cases} y^{-2\nu},& y\in (0,y_1],\\ y^{-\nu-\frac12} e^{-y}, & y>y_1,\end{cases} 
\eeq
for  some appropriate $y_1>0$ (in the sense of two-sided estimates with constants depending only on $\nu$ and $y_1$). Note that \eqref{g1.4} reads as
\beq
\label{g1.20}
\varphi(\tau) \simeq \tau^{\alpha/n-1},\quad \tau\in (0,T],
\eeq
in this case which implies that \eqref{g1.7} is true. In order to apply Theorem~\ref{theo-g-1.1} we need to verify \eqref{g1.5} and have found in this case that
\[
\eqref{g1.5}\quad\text{if, and only if,}\quad  
\tau^{\alpha/n-1}\in \widetilde{E}'(0,T),\quad
T\in\real_+.
\]
Recall that $\widetilde{E}'(0,T)$ is the restriction of $\widetilde{E}'(\real_+)$ to $(0,T)$.
\end{example}

\begin{example}\label{exm-g-1.4}
  Let $\Phi\in C^k(\real_+)$ satisfy assumptions \eqref{g1.3}--\eqref{g1.5} and \eqref{g1.16} for some $z_1\in\real_+$. Assume $T=V_n z_1^n$, and
  \beq
  \label{g1.21}
  \Phi(z)=z^{\alpha-n} \Lambda (z),\quad 0<\alpha<n,\quad z\in (0,z_1].
  \eeq
  Here $\Lambda \in C^k(0,z_1]$ is a positive function, with
    \beq
    \label{g1.22}
    z^j \Lambda^{(j)}(z)=\varepsilon_j(z)\Lambda(z), \quad \text{with}\quad \varepsilon_j(z)\xrightarrow[z\to 0+]{} 0,\quad j=1, \dots, k.
    \eeq
    Then $\Lambda$ is a slowly varying function on $(0,z_1]$, i.e., for all $\gamma>0$,
      \beq\begin{split}
      \label{g1.23}
      z^\gamma \Lambda(z)\quad \text{is monotonically increasing,}\\ \quad z^{-\gamma} \Lambda (z)\quad\text{is monotonically decreasing}.\end{split}
      \eeq
      We further conclude that
      \beq
      \label{g1.24}
      \varphi(\tau)=\tau^{\alpha/n-1} \lambda(\tau),\quad \lambda(\tau)=\Lambda\left(\left(\frac{\tau}{V^n}\right)^{1/n}\right),\quad \tau\in (0,T].
      \eeq
The function $\lambda$ is slowly varying as well on $(0,T]$, and \eqref{g1.7} is satisfied.       
      One verifies that $\Phi$ satisfies conditions \eqref{g1.15} and \eqref{g1.17} such that Theorem~\ref{theo-g-1.1} is applicable whenever \eqref{g1.5} is true, where $\varphi$ is given by \eqref{g1.24}. 
      
\begin{proof}
   
      We apply the Leibniz formula to $\Phi$ given by \eqref{g1.21} and get
      \begin{align*}
        \Phi^{(k)}(z)= &\ (\alpha-n)\cdots (\alpha-n-(k-1)) z^{\alpha-n-k}\Lambda(z) \\
        & + \sum_{j=1}^k C_{k,j} (\alpha-n)\cdots (\alpha-n-(k-j-1)) z^{\alpha-n-(k-j)}\Lambda^{(j)}(z),
      \end{align*}
      where $C_{k,j}$ are the binomial coefficients. Therefore, by \eqref{g1.22},
     \begin{align*}
       z^k\Phi^{(k)}(z)= &\ \Phi(z)\Big[ (-1)^k (n+k-1-\alpha)\cdots(n-\alpha)\\
         &+ \sum_{j=1}^k C_{k,j} (\alpha-n)\cdots (\alpha-n-(k-j-1)) \varepsilon_j(z)\Big].
      \end{align*}
      This assertion implies \eqref{g1.17} for sufficiently small $z_1>0$ since $\varepsilon _j(z)\to 0$ for $z\to 0+$, $j=1, \dots, k$. Note that
      \beq\label{g5.6}
      \left(z^{-1} \frac{\dint}{\dint z}\right)^m z^{\alpha-n} = (\alpha-n)\cdots (\alpha-n-2(m-1)) z^{\alpha-n-2m}.
      \eeq
      Furthermore, condition \eqref{g1.22} implies for $m=1, \dots, k$,
      \beq\label{g5.7}
      \left(z^{-1} \frac{\dint}{\dint z}\right)^m \Lambda(z)=\delta_m(z) z^{-2m} \Lambda(z),
      \eeq
      where $\delta_m(z)\to 0$ for $z\to 0+$. An analogue of Leibniz' rule together with \eqref{g5.6}, \eqref{g5.7} gives
      \begin{align*}
        & \left(z^{-1}\frac{\dint}{\dint z}\right)^l \Phi(z) \\
        & =  \ \sum_{j=0}^l C_{l,j}\left[\left(z^{-1}\frac{\dint}{\dint z}\right)^{l-j} z^{\alpha-n}\right] \left[\left(z^{-1} \frac{\dint}{\dint z}\right)^j \Lambda(z)\right]   \\
        & =   \ (\alpha-n)\cdots (\alpha-n-2(l-1))z^{\alpha-n-2l}\Lambda(z)  \\
        & \quad~ + \sum_{j=1}^l C_{l,j}  (\alpha-n)\cdots (\alpha-n-2(l-j-1))z^{\alpha-n-2(l-j)}\delta_j(z)z^{-2j}\Lambda(z)\\
          & =  \ \Phi(z) z^{-2l} F(z), 
        \end{align*}
        where $$F(z)=(\alpha-n)\cdots (\alpha-n-2(l-1))+ \sum_{j=1}^l C_{l,j}  (\alpha-n)\cdots (\alpha-n-2(l-j-1))\delta_j(z).$$
      Since the term $F(z)$ is bounded on $(0,z_1]$, we obtain \eqref{g1.15}. Thus Theorem~\ref{theo-g-1.1} can be applied. Finally $\varphi$ is determined by \eqref{g1.24} with $0<\alpha<n$, and $\lambda$ being slowly varying on $(0,T]$, and we conclude
        \beq
        \label{g5.8}
\int_0^t \varphi(\tau)\dint \tau = \int_0^t \lambda(\tau)\tau^{\alpha/n-1} \dint \tau \simeq \lambda(t) t^{\alpha/n} = \varphi(t) t,
        \eeq
        where the involved constants do not depend on $t\in (0,T)$, see also Remark~\ref{rem-g-5.1} below. Thus condition \eqref{g1.7} is satisfied and we have the equivalence that $ H^G_E(\rn)\hookrightarrow C(\rn)$ holds if, and only if,
        \[
        \tau^{\alpha/n-1} \lambda(\tau) \in \widetilde{E}'(0,T).
        \]
     
 \end{proof}   
\end{example}

\begin{remark}\label{rem-g-5.1}
  In \eqref{g5.8} we used some well-known properties of slowly varying functions $\lambda$, which are positive on $(0,T)$: for any $\gamma>0$,
    \begin{align}
    \label{g5.9}
\int_0^t \tau^{\gamma-1} \lambda(\tau)\dint\tau & \simeq t^\gamma \lambda(t),\\
    \label{g5.10}
\int_t^T \tau^{-\gamma-1} \lambda(\tau)\dint\tau &\leq c_\gamma t^{-\gamma} \lambda(t),\\
    \label{g5.11}
\lambda(t) & = {\mathbf o}\left(\int_t^T \tau^{-1} \lambda(\tau)\dint\tau\right)\quad\text{for}\quad t\to 0+.
   \end{align}

\end{remark}

Now we describe some important characteristic of the smoothness of functions from $H^{G}_{E},$ the uniform majorant for moduli of continuity $\Omega^{k}_{EG}(t^{1/n}),\, t \in (0, T),$ namely,
\beq
\label{g-ad2.47}
\Omega^{k}_{EG}(t^{1/n})=\sup \left\lbrace \omega_k(u, t^{1/n}):\, u \in H^{G}_{E}(\rn);\, \|u\|_{H^{G}_{E}}\leq 1\right\rbrace.
\eeq
In case of $k=1$ this is a variant of the continuity envelope function studied in \cite{Ha-crc,T-func} in general, and in \cite{GoHa} for Bessel potentials.

\begin{theorem}
	\label{theo-g-ad-2.11}
	Let $\Phi$ be a standard function with \eqref{g1.3}, $\varphi$, given by \eqref{g1.4}, satisfy \eqref{g1.5}.
	\bli
	\item[{\bfseries\upshape 1.}]
Under the assumptions of Theorem \ref{theo-g-1.1}, Part 1,  the following estimate holds:
	\beq
	\Omega^{k}_{EG}(t^{1/n})\leq c_{1}\|\Omega_{\varphi}(t, \cdot)\|_{\widetilde{E}'(0, T)},\quad t \in (0, T).
	 \label{g-ad2.48}
	\eeq

	\item[{\bfseries\upshape 2.}]
	Under the assumptions of Theorem \ref{theo-g-1.1}, Part 2, we have both estimates: \eqref{g-ad2.48} and 
	\beq
	\Omega^{k}_{EG}(t^{1/n})\geq c_{2}\|\Omega_{\varphi}(t, \cdot)\|_{\widetilde{E}'(0, T)},\quad t \in (0, T).
	\label{g-ad2.49}
	\eeq
	 The constants  $c_{1}, c_{2}$ are the same as in Theorem \ref{theo-g-1.1}.   
	\eli
\end{theorem}

\begin{proof} {\em Step 1}.\quad Let $u \in H^{G}_{E}(\rn)$.
	For any $f\in {E}(\rn)$ such that  $G\ast f=u$
	we have the estimate, see  \eqref{g-ad2.28},
	\begin{align*}
 \omega_k(u, t^{1/n})\leq c_{1}\int_{0}^{T}\Omega_{\varphi}(t,\tau)f^{\ast}(\tau)\dint \tau
& \leq c_{1}\|\Omega_{\varphi}(t, \cdot)\|_{\widetilde{E}'(0, T)}\|f^{\ast}\|_{\widetilde{E}(0, T)} \\
& \leq  c_{1}\|\Omega_{\varphi}(t, \cdot)\|_{\widetilde{E}'(0, T)}\|f\|_{E(\rn)}.
		\end{align*}
		Therefore,    
			\begin{equation*}
		\omega_k(u, t^{1/n})\leq  c_{1}\|\Omega_{\varphi}(t, \cdot)\|_{\widetilde{E}'(0, T)}\|u\|_{H^G_E(\rn)}, \quad\forall u \in H^G_E(\rn).
		\end{equation*}
		This yields \eqref{g-ad2.48}.\\
		
		{\em Step 2}.\quad
	Note that $\Omega_{\varphi}(t, \tau)\geq 0$  decreases as a function of $\tau \in (0,T),$ so that we have by the well-known formula for an associated norm in the RIS   $\widetilde{E}(0, T)$
		\begin{equation*}
	\|\Omega_{\varphi}(t, \cdot)\|_{\widetilde{E}'(0, T)}=\sup \left\lbrace \int_0^T \Omega_{\varphi}(t, \tau)\sigma(\tau)\dint \tau:\, \sigma \in \widetilde{E}_0(0, T);\, \|\sigma\|_{\widetilde{E}(0, T)}\leq 1\right\rbrace.
	\end{equation*}
	Thus, for every $\varepsilon \in (0,1)$  there exists $\sigma_{0, \varepsilon} \in \widetilde{E}_0(0, T)$   such that  $\|\sigma_{0, \varepsilon}\|_{\widetilde{E}(0, T)}\leq 1$, and
	\begin{equation*}
	\int_0^T \Omega_{\varphi}(t, \tau)\sigma_{0, \varepsilon}(\tau)\dint \tau \geq (1-\varepsilon)\|\Omega_{\varphi}(t, \cdot)\|_{\widetilde{E}'(0, T)}.
	\end{equation*}
	Let  $\sigma_{\varepsilon}\in \widetilde{E}_0(\real_+)$ be the extension of $\sigma_{0, \varepsilon}$   by zero from  $(0, T)$   onto  $\real_+$. Then, according to \eqref{g-ad2.29}, there exists $f_{\varepsilon} \in E(\rn)$ such that $f_{\varepsilon}^{\ast}\leq \sigma_{\varepsilon},\quad G\ast f_{\varepsilon}=u_{\varepsilon} \in H^{G}_{E}(\rn)$, 
	\beq
	\label{g-ad2.50}
	\omega_k(u_{\varepsilon}, t^{1/n})\geq c_{2}\int_{0}^{T}\Omega_{\varphi}(t,\tau)\sigma_{0,\varepsilon}(\tau)\dint \tau
	\geq (1-\varepsilon) c_{2}\|\Omega_{\varphi}(t, \cdot)\|_{\widetilde{E}'(0, T)}.
	\eeq
	 Moreover,
	 \beq
	 \label{g-ad2.51}
	\|u\|_{H^{G}_{E}(\rn)}\leq 	\|f_{\varepsilon}\|_{E(\rn)}=\|f_{\varepsilon}^{\ast}\|_{\widetilde{E}(\real_+)}\leq \|\sigma_{\varepsilon}\|_{\widetilde{E}(\real_+)}=\|\sigma_{0,\varepsilon}\|_{\widetilde{E}(0, T)}\leq 1.
	 \eeq
	 Consequently, by \eqref{g-ad2.47}, \eqref{g-ad2.50}, and \eqref{g-ad2.51}
	 \begin{equation*}
	 \Omega_{EG}^{k}(t^{1/n})\geq c_{2}\|\Omega_{\varphi}(t, \cdot)\|_{\widetilde{E}'(0, T)}.
	 \end{equation*}
	 This leads to \eqref{g-ad2.29} by passing to the limit for $\varepsilon \rightarrow 0.$
\end{proof}

\section{Embeddings into Calder\'on spaces}
\label{sect-2}

\subsection{Generalized Banach function spaces}
\label{sect-2-1}

We need some generalization of the notion of Banach function spaces (BFS), see Definition~\ref{bfs-ris}. Let $\mu$ denote the Lebesgue measure on $(0,T)$, $T\in (0,\infty]$.

\begin{definition}\label{gbfs-def}
A linear space $X=X(0,T)$ of measurable functions on $(0,T)$, equipped with the norm $\|\cdot\|_X$, is called a {\em generalized Banach function space}, shortly: {\em GBFS}, if the following conditions are satisfied:
\bit
\item[{\bfseries\upshape (P1)}]
$\|f\|_X=0 \iff f=0\ $ $\mu$-a.e. on $(0,T)$;
\item[{\bfseries\upshape (P2)}]
  $|f|\leq g,\  g\in X \quad\text{implies}\quad f\in X, \ \|f\|_X \leq \|g\|_X$;
\item[{\bfseries\upshape (P3)}]
If for some functions $f_n\in X$, $f_n\geq 0$, $n\in\nat$, and $f_n$ monotonically increasing to $f$, then $\|f_n\|_X \to \|f\|_X$ ~ for $n\to\infty$.
\item[{\bfseries\upshape (P4)}]
  For every measurable set $B\subset (0,T)$ with $\mu(B)>0$, there exists some $h_B>0$ $\mu$-a.e. in $B$, and some $c_B>0$, such that for all $f\in X$,
  \[
  \int_B h_B |f| \dint \mu \leq c_B \|f\|_X.
  \]
\item[{\bfseries\upshape (P5)}]
  For every measurable set $B\subset (0,T)$ with $\mu(B)>0$, there exists some $f_B\in X$ such that $f_B>0$ $\mu$-a.e. in $B$.
  \eit
\end{definition}

\begin{remark}
  Note that in \cite{BS} for a BFS it is required that $h_B=f_B=\chi_B$ in (P4), (P5), respectively, where $\chi_B$ is the characteristic function of $B$. Note that, if $X=X(0,T)$ is a BFS, and the function $\nu$ is $\mu$-measurable such that $0<\nu<\infty$ $\mu$-a.e. in $(0,T)$, then
  \[
  X_\nu = \left\{f: f \nu\in X, \ \|f\|_{X_\nu} = \|f\nu\|_X\right\}
  \]
is a GBFS. Moreover, a GBFS is in fact a Banach space, as well as its associated space. The duality principle holds for GBFS (the twice associated space coincides with the initial space, cf. \cite{BGZ}). 
  \end{remark}

Let $K=K(0,T)$ be some cone of non-negative $\mu$-measurable functions on $(0,T)$ equipped with the positively homogeneous functional $\varrho_K$. Recall that for the GBFS $X=X(0,T)$ the embedding $ K \mapsto X$ means that $K\subset X$ and
\beq
\label{g2.1}
\exists\ c=c_K>0: \quad \|h\|_X \leq c_K\varrho_K(h)\quad\text{for all}\ h\in K.
\eeq

\begin{definition}\label{def-g-2.2}
A GBFS $X_0=X_0(0,T)$ is called {\em optimal} for the embedding $K\mapsto X$, if
\bli
\item[{\bfseries\upshape (i)}]
  $K\mapsto X_0$,
\item[{\bfseries\upshape (ii)}]
  whenever $K\mapsto Y$, where $Y$ is a GBFS, then this implies $X_0\subset Y$.
  \eli
\end{definition}

\begin{remark}\label{rem-g-2.3}
  Let $K$ and $M$ be some cones of non-negative $\mu$-measurable functions on $(0,T)$ equipped with the functionals $\varrho_K$ and $\varrho_M$. If $K \approx M$, then
\bli
\item[{\upshape (1)}]
  for every GBFS $X=X(0,T)$ we have $K\mapsto X$ if, and only if, $M\mapsto X$, and the ratio of the constants $c_K/c_M$ in \eqref{g2.1} depends on the constants of the mutual coverings of the cones only, see \eqref{g1.13}, \eqref{g1.14},
\item[{\upshape (2)}]
  a GBFS $X_0=X_0(0,T)$ is optimal for both the embeddings $K\mapsto X$ as well as $M\mapsto X$.
  \eli
  This can be seen as follows. Let us show that whenever $M\prec K$ and $K\mapsto X$, then $M\mapsto X$.

  For every $h_1\in M$ we can find $h_2\in K$ such that
  \[h_1\leq h_2\quad\text{on}\quad (0,T)\quad\text{and}\quad \varrho_K(h_2)\leq c_0\varrho_M(h_1).\]
  Now $K\mapsto X$ implies $h_2\in X$ and, by property (P2) of Definition~\ref{gbfs-def}, $h_1\in X$ with $\|h_1\|_X\leq \|h_2\|_X\leq c_K\varrho_K(h_2)$. This finally leads to
  \[
  \|h_1\|_X\leq c_K c_0\varrho_M(h_1)\quad\text{for all}\quad h_1\in M.
  \]
  But this is nothing else than $M\mapsto X$. Consequently, the equivalence $M \approx K$ implies the equivalence
  \[
  K\mapsto X \iff M\mapsto X.
  \]
Thus the same GBFS $X_0=X_0(0,T)$ is optimal for both embeddings $K\mapsto X$ and $M\mapsto X$.  
\end{remark}

Let $X=X(0,T)$ be a GBFS and $k\in\nat$. We introduce the Calder\'on space $\Lambda^k(C,X)$ (see for example \cite{goldman-8}, a more special version was considered in \cite{GNO-7}) as follows:
\begin{align}
\Lambda^k(C;X) = & \left\{u\in C(\rn): \ \omega_k(u; t^{1/n}) \in X(0,T)\right\},
\label{g2.2}\\
& \|u\|_{\Lambda^k(C,X)} = \|u\|_C + \|\omega_k(u; t^{1/n})\|_{X(0,T)}.
\label{g2.3}
\end{align}
The following non-trivial conditions hold:
\begin{align}
  \label{g2.4}
  \Lambda^k(C;X) \neq \{0\} & \iff \left\|t^{k/n}\right\|_{X(0,T)} < \infty, \\
  \label{g2.5}
  \Lambda^k(C;X) \neq C(\rn) & \iff \left\| 1 \right\|_{X(0,T)} = \infty.
\end{align}
Moreover, it is obvious that
\beq
\label{g2.6}
X_0(0,T)\hookrightarrow X(0,T) \quad\text{implies}\quad \Lambda^k(C;X_0) \hookrightarrow \Lambda^k(C;X).
\eeq

Now we are able to formulate a criterion for the embedding $H^G_E(\rn) \hookrightarrow \Lambda^k(C;X)$.

\begin{theorem}
  \label{theo-g-2.4}
  Let the conditions of Theorem~\ref{theo-g-1.1} be satisfied, and let $K$ be the cone given by \eqref{g1.10} with \eqref{g1.12}. Then
\beq
\label{g2.7}
H^G_E(\rn) \hookrightarrow \Lambda^k(C;X),
\eeq
if, and only if,
\beq
\label{g2.8}
K\mapsto X.
\eeq
The norm of the embedding operator in \eqref{g2.7} depends only on $k$, $n$, $T$, $a_1$, $a_2$, $z_1$, $\delta_1$ and on the norms of the embedding operators in \eqref{g1.7} and \eqref{g2.1}.
\end{theorem}

\begin{proof}
  First we show that under the condition \eqref{g1.5} we have the equivalence
  \beq
\eqref{g2.7} \iff M\mapsto X(0,T),
  \label{g6.1}
  \eeq
  where $M$ is the cone in \eqref{g1.8}. Indeed, in this case
  \[
  \|u\|_C\leq c_1 \|u\|_{H^G_E}\quad\text{for all}\quad u\in H^G_E(\rn),
  \]
  and \eqref{g2.7} is equivalent to
  \[
  \|\omega_k(u;t^{1/n})\|_{X(0,T)} \leq c_2\ \|u\|_{H^G_E}, \quad u\in H^G_E(\rn).
  \]
  This means that for $h\in M$,
  \[
\|h\|_{X(0,T)} \leq c_2\|u\|_{H^G_E},
  \]
  for every $u\in H^G_E(\rn)$ such that $\omega_k(u;t^{1/n}) = h(t)$, $t\in (0,T)$. Therefore, for $h\in M$,
  \[
  \|h\|_{X(0,T)}\leq c_2\inf\left\{\|u\|_{H^G_E}: u\in H^G_E(\rn), \ \omega_k(u;t^{1/n})=h(t)\right\},
  \]
  that is,
  \[
  \|h\|_{X(0,T)} \leq c_2 \varrho_M(h),\quad h\in M.
  \]
  This is equivalent to the embedding $M\mapsto X(0,T)$.
  
Now the equivalence of \eqref{g2.7} and \eqref{g2.8} follows from \eqref{g6.1}, \eqref{g1.14} due to Remark~\ref{rem-g-2.3}.
 \end{proof}

\begin{corollary}\label{cor-g-2.5}
  Let $X_0=X_0(0,T)$ be an optimal GBFS for the embedding \eqref{g2.8}, where $K$ is again the cone given by \eqref{g1.10}--\eqref{g1.12}. Then $\Lambda^k(C;X_0)$ is an optimal Calder\'on space for the embedding \eqref{g2.7}, that is,
  \beq
  \begin{split}
    H^G_E(\rn) \hookrightarrow \Lambda^k(C;X_0),\quad \text{and}\\
 \eqref{g2.7} \quad\text{implies}\quad \Lambda^k(C;X_0)\hookrightarrow \Lambda^k(C;X).
  \end{split}
\label{g2.9}  \eeq
\end{corollary}

\begin{proof}
  For the GBFS $X_0=X_0(0,T)$ we have
  \beq
  \label{g6.2}
K\mapsto X_0 \implies H^G_E(\rn) \hookrightarrow \Lambda^k(C; X_0),
  \eeq
  by Theorem~\ref{theo-g-2.4}. Assume that the embedding \eqref{g2.7} is true. Then $K\mapsto X$ implies $X_0\hookrightarrow X$ by \eqref{g2.8}, and by the definition of the optimal GBFS for the embedding $K\mapsto X$. Now we apply \eqref{g2.6} and obtain both the assertions in \eqref{g2.9}.
\end{proof}

\section{The description of the optimal Calder\'on space, I}
\label{sect-3}
We shall exemplify the results of Sections~\ref{sect-1} and \ref{sect-2} for the case when the basic RIS $E(\rn)$ coincides with a weighted Lorentz space, $E(\rn)=\Lambda_q(v)$, $1\leq q<\infty$, where the weight $v>0$ is a measurable function, recall Example~\ref{ex-Lorentz}, in particular \eqref{1.5}. Then $\Lambda_q(v)$, $1\leq q<\infty$, is equipped with the functional
\beq
\|f\|_{\Lambda_q(v)}  =\left(\int _0^\infty f^\ast\left(t\right)^{q}v\left(t\right) \dint t\right)^{1/q},\quad 1\le q<\infty.
\label{g3.1}
\eeq
General properties of Lorentz spaces can be found, for instance, in \cite{C-S,C-P-S-S}. Recall that
\beq
\label{g3.2}
\Lambda_q(v)\neq \{0\} \iff V(t)=\int_0^t v(\tau)\dint\tau <\infty,\quad t>0.
\eeq
Expression \eqref{g3.1} is equivalent to some norm if $q=1$ and $t^{-1}V(t)$ almost decreases, or, in case $q>1$, if there exists some $c>0$, such that
\beq
\label{g3.3}
t^q \int_t^\infty \tau^{-q}v(\tau)\dint \tau \leq c V(t),\quad t>0.
\eeq

We shall assume in the sequel that these conditions are satisfied. We further need the following notation, where $T>0$ and $t\in (0,T]$,
\begin{align}
W(t) = & V(t)^{-1}\int_0^t \varphi(\tau)\dint\tau,
\label{g3.4} \\
\Psi_q(t) = & \begin{cases}
\sup_{\tau\in (0,t]} W(\tau), & q=1, \\ \left(\int_0^t W^{q'}(\tau) v(\tau)\dint\tau\right)^{1/q'}, & 1<q<\infty, 
\end{cases}
\label{g3.5}
\end{align}
where $q'$ is defined as usual, $\frac1q+\frac{1}{q'}=1$, $1<q<\infty$. 
  
\begin{lemma}
  \label{lemma-g-3.1}
Let $T>0$, $1\leq q<\infty$.   Using the above notation,
  \beq
  \eqref{g1.5} \quad\text{is true\quad if, and only if,}\quad \Psi_q(T)<\infty.
  \label{g3.7}
  \eeq
\end{lemma}

\begin{proof}
  Let us consider the case of basic RIS
  \beq
  \label{g7.1}
  E(\rn)=\Lambda_q(v),\quad 1\leq q<\infty,
  \eeq
  see \eqref{g3.1}, \eqref{g3.2}. We define for $t>0$, $1<q<\infty$,
  \beq
  w(t)=V(t)^{-q'} v(t).
  \label{g7.2}
  \eeq
 Note that
  \beq
  \label{g7.3}
  \int_T^\infty w(t)\dint t = \frac{1}{q'-1}\left(V(T)^{1-q'}-\lim_{t\to\infty}V(t)^{1-q'}\right).
  \eeq

  We use the well-known description of the associated RIS for Lorentz spaces \eqref{g7.1} in the Luxemburg representation:
\[
\|\varphi_0\|_{\widetilde{E}'(\real_+)} = \begin{cases} \sup_{\tau\in\real_+} \frac{1}{V(\tau)} \int_0^\tau \varphi_0^*(s)\dint s,& q=1,\\
\left(\int_0^\infty \left(\int_0^\tau \varphi_0^*(s)\dint s\right)^{q'} w(\tau)\dint\tau\right)^{1/q'},& q>1,\end{cases}
\]
where $\varphi_0^\ast$ is the decreasing rearrangement of the function $\varphi_0:\real_+\to[0,\infty]$. For $\varphi$ given by \eqref{g1.4} we set
\[
\varphi_0(\tau)=\begin{cases} \varphi(\tau), &\tau\in (0,T),\\
0,& \tau\geq T.\end{cases}
\]
Since $\varphi\geq 0$ is monotonically decreasing and right-continuous, we conclude
\[\varphi_0^\ast(s) = \varphi(s)\chi_{(0,T)}(s),
\]
and thus
\[ \int_0^\tau \varphi_0^\ast (s)\dint s = \left(\int_0^\tau\varphi(s)\dint s\right)\chi_{(0,T)}(\tau) + \left(\int_0^T\varphi(s)\dint s\right)\chi_{[T,\infty)}(\tau). 
  \]
  In view of $\ \|\varphi\|_{\wt{E}'(0,T)} = \|\varphi_0\|_{\wt{E}'(\real_+)}\ $ this leads to
  \[
  \|\varphi\|_{\wt{E}'(0,T)} = \max\left\{\sup_{\tau\in (0,T)} W(\tau), \left(\int_0^T \varphi(s)\dint s\right) \sup_{\tau\geq T} V(\tau)^{-1}\right\} = \Psi_1(T)
  \]
  in case of $q=1$, see \eqref{g3.4} and \eqref{g3.5}. In case of $q>1$, using that $q'=\frac{q}{q-1}$, we conclude from \eqref{g3.4}, \eqref{g3.5}, \eqref{g7.2} and \eqref{g7.3} that
  \begin{align*}
    \|\varphi\|_{\wt{E}'(0,T)} \simeq & \left( \int_0^T W(\tau)^{q'} v(\tau)\dint \tau\right)^{1/q'} + \left(\int_0^T \varphi(s)\dint s\right) \left(\int_T^\infty w(\tau)\dint \tau\right)^{1/q'} \\
    = & \ \Psi_q(T) + \left(\int_0^T \varphi(s)\dint s\right) \frac{1}{(q'-1)^{1/q'}}\left(V(T)^{1-q'}-\lim_{t\to\infty}V(t)^{1-q'}\right)^{1/q'}.
  \end{align*}
Thus the equivalence \eqref{g3.7} is shown, as the second term is finite in view of \eqref{g1.3}, \eqref{g1.4}.  
\end{proof}

In view of the description of the optimal Calder\'on spaces we introduce two alternative collections of the conditions for $\varphi$ and $v$. For that reason we complement our above notation \eqref{g3.4}, \eqref{g3.5} as follows:
\begin{align}
\widetilde{W}(t) = & V(t)^{-1} t^{1-\frac{k}{n}} \varphi(t),
\label{g3.12}\\
U_q(t) = & \begin{cases}
\sup_{\tau\in [t,T]} \widetilde{W}(\tau), & q=1, \\ \left(\int_t^T \widetilde{W}^{q'}(\tau) v(\tau)\dint\tau\right)^{1/q'}, & 1<q<\infty, 
\end{cases}
\label{g3.13}
  \end{align}
  where again $t\in (0,T]$ is assumed. Note that $\widetilde{W}$ is a continuous bounded function on $[t,T]$ for any $t\in (0,T]$, such that the expressions in \eqref{g3.13} are well-defined.

    Now we can formulate the alternative assumptions.
\bli
\item[{\bfseries\upshape (A)}]
  There exists a constant $d_1>0$ such that for every $t\in (0,T ]$,
    \beq
    \label{g3.8}
\int_t^T \tau^{-\frac{k}{n}} \varphi(\tau)\dint \tau \leq d_1 t^{-\frac{k}{n}} \int_0^t \varphi(\tau)\dint\tau,
    \eeq
    and, in addition,
\beq
\label{g3.9}
\exists\ \varepsilon>0: \ t^\varepsilon V(t)^{-1}\quad \text{is monotonically decreasing for}\quad t\in (0,T].
\eeq
\item[{\bfseries\upshape (B)}]
  There exists a constant $d_2>0$ such that for every $t\in (0,T]$
    \beq
    \label{g3.10}
\int_0^t \tau^{-\frac{k}{n}} \varphi(\tau)\dint \tau \leq d_2 t^{1-\frac{k}{n}} \varphi(t),
    \eeq
    and, in addition,
\beq
\label{g3.11}
\exists\ \varepsilon>0: \ t^\varepsilon U_q(t)\quad \text{is monotonically decreasing for}\quad t\in (0,T].
\eeq
 \eli

 \begin{remark}\label{rem-g-3.2}
   Let $\lambda>0$ be slowly varying on $(0,T]$,
     \beq
     \label{g3.15}
     0<\alpha<n,\quad \varphi(t)=t^{\frac{\alpha}{n}-1} \lambda(t),\quad t\in (0,T],
     \eeq
     similar to \eqref{g1.24}. Recall that $\lambda\equiv 1$ corresponds to the classical Bessel potentials. Then \eqref{g3.8} holds if, and only if, $\alpha<k$, whereas \eqref{g3.10} holds, if, and only if, $\alpha>k$. \\

This can be seen as follows. Let $\varphi$ be given by \eqref{g3.15} and denote
     \begin{align}
       \label{g7.4}
       A(t)= & \int_t^T \tau^{-k/n}\varphi(\tau)\dint\tau = \int_t^T \tau^{\frac{\alpha-k}{n}-1} \lambda(\tau)\dint\tau,\\
B(t)= & \ t^{-k/n} \int_0^t\varphi(\tau)\dint\tau = t^{-k/n} \int_0^t \tau^{\alpha/n-1} \lambda(\tau)\dint\tau.
       \label{g7.5}
     \end{align}
     According to \eqref{g5.9} for $\alpha>0$,
     \beq
     \label{g7.6}
B(t)\simeq t^{\frac{\alpha-k}{n}} \lambda(t),\quad t\in (0,T),
     \eeq
     and by \eqref{g5.10} for $0<\alpha<k$,
     \[
     A(t)\leq c\ t^{\frac{\alpha-k}{n}}\lambda(t) \simeq B(t),\quad t\in (0,T),
     \]
     such that \eqref{g3.8} follows. If $\alpha=k$, then \eqref{g7.4}, \eqref{g7.6} and \eqref{g5.11} show that
     \[
     B(t) \simeq \lambda(t) \simeq {\mathbf o}\left(A(t)\right),\quad t\to 0+,
     \]
     such that \eqref{g3.8} fails. If $\alpha>k$, then \eqref{g7.6} and \eqref{g5.9} show that
     \[
     B(0+)=0,\quad A(0+) = \int_0^T \tau^{\frac{\alpha-k}{n}-1}\lambda(\tau)\dint\tau \simeq T^{\frac{\alpha-k}{n}}\lambda(T)>0,
     \]
     so that \eqref{g3.8} fails. Therefore
     \eqref{g3.8} holds if, and only if, $\alpha<k$.

     Next we show that \eqref{g3.10} holds if, and only if, $\alpha>k$. For a function $\varphi$ given by \eqref{g3.15} we denote by
     \[
     C(t)=\int_0^t \tau^{-\frac{k}{n}}\varphi(\tau)\dint\tau = \int_0^t \tau^{\frac{\alpha-k}{n}-1} \lambda(\tau)\dint\tau.
     \]
     If $\alpha<k$, then $C(t)=\infty$, $t\in (0,T)$, for every slowly varying function $\lambda>0$. Thus \eqref{g3.10} fails. If $\alpha=k$, then
     \[
     t^{1-\frac{k}{n}} \varphi(t)= \lambda(t)={\mathbf o}\left(\int_0^t \lambda(\tau) \frac{\dint\tau}{\tau}\right) \quad \text{for}\quad t\to 0+
     \]
     for every slowly varying function $\lambda>0$. Hence \eqref{g3.10} fails as well. Finally, in the remaining case $\alpha>k$, we have by \eqref{g5.9} that
     \[
     C(t)\simeq t^{\frac{\alpha-k}{n}} \lambda(t) = t^{1-\frac{k}{n}} \varphi(t), \quad t\in (0,T),
     \]
and \eqref{g3.10} holds.
 \end{remark}

 Now we present one of our main results. Its proof however has to be postponed to Section~\ref{sect-3b} as we shall need some detailed preparation. But here we want to formulate the corresponding result first and collect some further consequences and examples below.

 We introduce the following notation,
\beq
\label{g3.16}
\left\|f\right\|_{\mathring{X}_0} = \left(\int_0^T \left(\frac{\|f\|_{L_\infty(0,t)}}{\Psi_q(t)}\right)^q \frac{\dint \Psi_q(t)}{\Psi_q(t)}\right)^{1/q},
\eeq
with $\Psi_q$ defined by \eqref{g3.5}. Let $T_1\in (0,T)$ be such that $\Psi_q(T_1)=\frac12 \Psi_q(T)$.

\begin{theorem}
  \label{theo-g-3.3}
  Let $1\leq q<\infty$, $T>0$, and assume that $\Psi_q(T)<\infty$, where $\Psi_q$ is given by \eqref{g3.5}. Let the cone $K$ be given by \eqref{g1.10} with \eqref{g1.12}. 
  Assume that at least one of the above conditions {\upshape\bfseries (A)} or {\upshape\bfseries (B)}, given by \eqref{g3.8}--\eqref{g3.11}, is satisfied. Then the optimal GBFS $X_0=X_0(0,T)$ for the embedding $K\mapsto X$ has the following norm:
  \bli
\item[{\upshape\bfseries (i)}]
  if $q=1$ and $\Psi_1(0+)=\lim_{t\downarrow 0} \Psi_1(t) >0$, then
  \beq
  \label{g3.17}
  \|f\|_{X_0} = \|f\|_{L_\infty(0,T)},
  \eeq
\item[{\upshape\bfseries (ii)}]
  if $q=1$ and $\Psi_1(0+)=\lim_{t\downarrow 0} \Psi_1(t) =0$, or $1<q<\infty$, then 
  \beq
  \label{g3.18}
  \|f\|_{X_0} = \|f\|_{\mathring{X}_0} + \Psi_q(T)^{-1} \| f\|_{L_\infty(T_1,T)}.
  \eeq
\eli
\end{theorem}


\section{Construction of the associated norm to the optimal norm}\label{sect-3a}

\subsection{Some preparation: General description}\label{sect-g-7}

Let the assumptions of Theorem~\ref{theo-g-1.1} be satisfied, and $K$ the cone given by \eqref{g1.10} with \eqref{g1.12}. We will show that the results of \cite[Thm.~3.1, Rem.~3.2, Ex.~3.3]{BGZ} are applicable here. Note that here $A=D=(0,T)$, $\mu=\nu$ are Lebesgue measures, $\Omega_\varphi(t,\tau)$ is determined by \eqref{g1.12}, so that for $t,\tau\in (0,T)$,
\beq
\label{g7.7}
\Omega_\varphi(t,\tau) \simeq\begin{cases}\varphi(\tau), & \tau\in (0,t], \\ t^{k/n} \tau^{-k/n}\varphi(\tau), & \tau>t.
\end{cases}
\eeq
To apply \cite[Thm.~3.1]{BGZ} it is sufficient to verify that
\beq
\label{g7.8}
c_0=\left\|\int_0^T \Omega_\varphi(t,\cdot)\dint t\right\|_{\widetilde{E}'(0,T)} < \infty,
\eeq
and the existence of some $\sigma_0\in \widetilde{E}_0(0,T)$ such that
\beq
\label{g7.9}
\int_0^T \Omega_\varphi(t,\tau)\sigma_0(\tau)\dint\tau>0, \quad t\in (0,T),
\eeq
as these conditions imply \cite[(3.7), (3.8)]{BGZ}. According to \eqref{g7.7}, for every $g\in M^+(0,T)$,
\beq
\label{g7.10}
\int_0^T \Omega_\varphi(\xi,\tau) g(\xi)\dint\xi \simeq \tau^{-k/n} \varphi(\tau) \int_0^\tau \xi^{k/n}g(\xi)\dint\xi + \varphi(\tau)\int_{\tau}^T g(\xi)\dint\xi.
\eeq
Let $g(t)\equiv 1$, $t\in (0,T)$, then for $\tau\in (0,T)$,
\[
\int_0^T \Omega_\varphi(t,\tau) \dint t \simeq \tau\varphi(\tau) + (T-\tau)\varphi(\tau)=T\varphi(\tau),
\]
so that
\[
c_0=\left\|\int_0^T\Omega_\varphi(t,\cdot)\dint t\right\|_{\widetilde{E}'(0,T)} \simeq \left\|\varphi\right\|_{\widetilde{E}'(0,T)}  < \infty,
\]
because of assumption \eqref{g1.5}. Furthermore, with
\[
\sigma_0 = \chi_{(0,T)}\in\widetilde{E}_0(0,T),
\]
we obtain in view of \eqref{g7.7} for every $t\in (0,T)$,
\[
\int_0^T \Omega_\varphi(t,\tau)\sigma_0(\tau)\dint\tau = \int_0^T \Omega_\varphi(t,\tau)\dint\tau \geq c\int_0^t \varphi(\tau)\dint\tau>0,
\]
which implies \eqref{g7.9}. Finally, an application of \cite[Thm.~3.1, Rem.~3.2, Ex.~3.3]{BGZ} shows that the associated norm to the optimal one for the embedding $K\mapsto X$ coincides with
\beq
\label{g7.11}
\varrho_0(g)=\left\|\int_0^T\Omega_\varphi(t,\cdot)g(t)\dint t\right\|_{\widetilde{E}'(0,T)} ,\quad g\in M^+(0,T).
\eeq

\subsection{The case $E=\Lambda_1(v)$}\label{sect-g-8}

\begin{lemma}
  Let the assumptions \eqref{g1.3}, \eqref{g1.4} and \eqref{g3.1}-\eqref{g3.3} be satisfied with $q=1$. Then the following estimate holds for the norm \eqref{g7.11},
    \beq
    \varrho_0(g)\simeq \wt{\varrho}_0(g)+\varrho_1(g), \quad g\in M^+(0,T),
    \label{g8.1}
\eeq
where
\begin{align}
  \wt{\varrho}_0(g) & = \sup_{t\in (0,T)} \left\{ V(t)^{-1} \left(\int_0^t\varphi(\tau)\dint\tau\right)\left(\int_t^T g(\xi)\dint\xi\right)\right\},    \label{g8.2}\\
    \varrho_1(g)& = \sup_{t\in (0,T)}  \left\{V(t)^{-1} \left(\int_0^t
  \Phi_k(\xi,t)g(\xi)\dint\xi\right)\right\},
    \label{g8.3}
\end{align}
where
\beq
\label{g8.4}
\Phi_k(\xi,t)=\int_0^\xi \varphi(\tau)\dint\tau + \xi^{k/n} \int_\xi^t \tau^{-k/n} \varphi(\tau)\dint\tau.
\eeq
\label{lemma-g-8.1}
  \end {lemma}

\begin{proof}
  For $g\in M^+(0,T)$ we define
  \beq
  \Psi_0(g,\tau) = \begin{cases}
\int_0^T \Omega_\varphi(\xi,\tau) g(\xi) \dint\xi, & \tau\in (0,T), \\ 0, &  \tau\geq T.
  \end{cases}
    \label{g8.5}
  \eeq
  Then, according to \eqref{g7.11},
  \beq
  \varrho_0(g)=\left\|\Psi_0(g)\right\|_{\wt{E}'(\real_+)}.
  \label{g8.6}
  \eeq
  In our setting, $\wt{E}(\real_+)=\Lambda_1(v)$ implies
  \beq
  \wt{E}'(\real_+) = M_V(\real_+),
  \label{g8.7}
  \eeq
  where $M_V$ is the Marcinkiewicz space normed by
  \[
{ \|f\|_{M_V} = \sup_{t>0} V(t)^{-1} \int_0^t f^\ast(\tau)\dint\tau},
  \]
recall \eqref{1.6} in Example~\ref{ex-Lorentz}.
  and $f^\ast$ denotes the decreasing rearrangement of $f$, as usual. Since $0\leq f(\tau)=\Psi_0(g,\tau)$ is decreasing and right-continuous, hence $f^\ast(\tau)=\Psi_0(g,\tau)$. Thus \eqref{g8.6} implies
  \beq
  \varrho_0(g)=\sup_{t>0} V(t)^{-1} \int_0^t \Psi_0(g,\tau)\dint\tau,
  \label{g8.8}
  \eeq
  which in view of \eqref{g8.5} leads to
  \begin{align*}
    \varrho_0(g) &= \max\left\{\sup_{t\in (0,T)} \left(V(t)^{-1} \int_0^t \Psi_0(g,\tau)\dint\tau\right); \left(\int_0^T \Psi_0(g,\tau)\dint\tau\right) \sup_{t\geq T} V(t)^{-1}\right\}\\
    &= \sup_{t\in (0,T]} \left(V(t)^{-1} \int_0^t \Psi_0(g,\tau)\dint\tau\right) \\
    & \simeq \sup_{t\in (0,T]} \left(V(t)^{-1} \int_0^t \int_0^T \Omega_\varphi(\xi,\tau) g(\xi) \dint\xi\dint\tau\right).
    \end{align*}
  We substitute \eqref{g7.10} into the last formula and obtain
  \[
  \varrho_0(g) \simeq \wt{\varrho}_0(g) + \varrho_1(g),
  \]
  where $\wt{\varrho}_0(g)$ is determined by \eqref{g8.2} and
  \beq
  \varrho_1(g)= \sup_{t\in (0,T]} \left(V(t)^{-1} \int_0^t \left(\int_\tau^t g(\xi)\dint\xi + \tau^{-k/n} \int_0^\tau \xi^{k/n} g(\xi)\dint\xi\right) \varphi(\tau)\dint\tau\right).
  \label{g8.9}
  \eeq
Changing the order of integration, equality \eqref{g8.9} gives \eqref{g8.3}.
\end{proof}

\begin{remark}
\label{rem-g-8.2}
Let the conditions of Lemma~\ref{lemma-g-8.1} be satisfied and
\[
\sup_{t\in (0,T)} V(t)^{-1} \int_0^t \varphi(\tau)\dint\tau = \infty.
\]
Then $g\in M^+(0,T)$ with $\varrho_0(g)<\infty$ implies $g=0$ a.e. on $(0,T)$. This is a consequence of \eqref{g8.2}.
\end{remark}

\begin{corollary}
  \label{cor-g-8.3}
  Let the conditions of Lemma~\ref{lemma-g-8.1} be satisfied and the following estimate be valid for $\xi\in (0,T)$,
  \beq
  \xi^{k/n}\int_\xi^T \tau^{-k/n}\varphi(\tau)\dint\tau \leq d_1\int_0^\xi \varphi(\tau)\dint\tau,
  \label{g8.10}
  \eeq
  for some $d_1\in\real_+$ not depending on $\xi$. Then
  \beq
  \varrho_0(g)\simeq \widehat{\varrho}_0(g)=\sup_{t\in(0,T]} \left\{V(t)^{-1} \left(\int_0^t \left(\int_\tau^T g(\xi)\dint\xi\right)\varphi(\tau)\dint\tau\right)\right\}.
  \label{g8.11}
  \eeq
\end{corollary}

\begin{proof}
  Note that in this case \eqref{g8.3} and \eqref{g8.10} yield
  \begin{align}
    \varrho_1(g)&\simeq \sup_{t\in (0,T]} \left\{ V(t)^{-1} \int_0^t \left(\int_0^\xi \varphi(\tau)\dint\tau\right)g(\xi)\dint\xi\right\}\nonumber\\
      &= \sup_{t\in(0,T]} \left\{ V(t)^{-1} \int_0^t \left(\int_\tau^t g(\xi)\dint\xi\right)\varphi(\tau)\dint\tau\right\}.
        \label{g8.12}
    \end{align}
We insert this identity in \eqref{g8.1}, take into account \eqref{g8.2} and arrive at \eqref{g8.11}.
\end{proof}

\begin{remark}
\label{rem-g-8.4}
If $k>n$, then the estimate \eqref{g8.10} is valid for every positive function $\varphi$ which decreases on $(0,T)$. Indeed, in this case
\begin{align*}
  \xi^{k/n} \int_\xi^T \tau^{-k/n} \varphi(\tau)\dint\tau &\leq \xi^{k/n} \varphi(\xi) \int_\xi^\infty \tau^{-k/n}\dint\tau = \xi\varphi(\xi) \left(\frac{k}{n}-1\right)^{-1} \\
  &\leq \left(\int_0^\xi \varphi(\tau)\dint\tau\right)\left(\frac{k}{n}-1\right)^{-1}.
\end{align*}
\end{remark}

\begin{corollary}
  \label{cor-g-8.5}
  If the conditions of Lemma~\ref{lemma-g-8.1} are satisfied, and the following estimate takes place for $\xi\in (0,T)$,
  \beq
  \int_0^\xi \tau^{-k/n}\varphi(\tau)\dint\tau\leq d_2 \xi^{1-k/n} \varphi(\xi),
  \label{g8.13}
  \eeq
  with $d_2\in\real_+$ not depending on $\xi$, then
  \beq
  \varrho_0(g) \simeq \wt{\varrho}_0(g) + \widehat{\varrho}_1(g),\quad g\in M^+(0,T),
  \label{g8.14}
  \eeq
  where $\wt{\varrho}_0$ is given by \eqref{g8.2} and
  \begin{align}
    \widehat{\varrho}_1(g) & =\sup_{t\in (0,T]} \left(V(t)^{-1} t^{1-k/n} \varphi(t) \int_0^t \xi^{k/n} g(\xi) \dint\xi\right) \nonumber\\
      &= \sup_{t\in (0,T]} \left( U_1(t) \int_0^t \xi^{k/n} g(\xi)\dint\xi\right),
    \label{g8.15}
  \end{align}
and $U_1$ is defined by \eqref{g3.13}.
\end{corollary}

\begin{proof}
  In this case, by \eqref{g8.13},
  \[
\int_0^\xi \varphi(\tau)\dint\tau \leq \xi^{k/n}\int_0^\xi \tau^{-k/n}\varphi(\tau)\dint\tau \simeq \xi\varphi(\xi)\leq \int_0^\xi \varphi(\tau)\dint\tau,
  \]
  so that
  \beq
  \xi\varphi(\xi) \simeq \int_0^\xi \varphi(\tau)\dint\tau \simeq \xi^{k/n}\int_0^\xi \tau^{-k/n}\varphi(\tau)\dint\tau,
  \label{g8.17}
  \eeq
  and
  \begin{align}
    \int_0^\xi \varphi(\tau)\dint\tau + \xi^{k/n} \int_\xi^t \tau^{-k/n}\varphi(\tau)\dint\tau & \simeq \xi^{k/n}\int_0^t \tau^{-k/n}\varphi(\tau)\dint\tau \nonumber\\
    &\simeq \xi^{k/n} t^{1-k/n} \varphi(t). \label{g8.18}
  \end{align}
  From here, and from \eqref{g8.3}-\eqref{g8.4} it follows that
  \beq
  \varrho_1(g)\simeq \sup_{t\in (0,T]} \left(V(t)^{-1} t^{1-k/n} \varphi(t) \left(\int_0^t \xi^{k/n} g(\xi)\dint\xi\right)\right) = \widehat{\varrho}_1(g).
  \label{g8.19}
  \eeq
  Therefore, \eqref{g8.1} implies \eqref{g8.14}. The second equality in \eqref{g8.15} with $U_1$ from \eqref{g3.13} is a well-known consequence of the fact that
  \beq
  0\leq \sigma_k(t)=\int_0^t \xi^{k/n}g(\xi)\dint\xi
  \label{g8.20}
  \eeq
increases in $t\in (0,T]$.  
\end{proof}

\begin{remark}
\label{rem-g-8.6}
For a positive decreasing function $\varphi$ the estimate \eqref{g8.13} is possible only when $k<n$. Otherwise the integral diverges. For such a function $\varphi$ the inverse estimate is evident, such that \eqref{g8.13} implies
\beq
\int_0^\xi \tau^{-k/n} \varphi(\tau)\dint\tau \simeq \xi^{1-k/n} \varphi(\xi).
\label{g8.16}
\eeq
\end{remark}

\begin{lemma}
\label{lemma-g-8.7}
Let the assumptions of Lemma~\ref{lemma-g-8.1} be satisfied and
\beq
\label{g8.21}
\int_0^t V(\xi)\xi^{-1}\dint\xi \leq c V(t),\quad t\in (0,T),
\eeq
for some $c>0$ independent of $t$. Then
\beq
\widehat{\varrho}_0(g) \simeq \wt{\varrho}_0(g),
\label{g8.22}
\eeq
using the notation \eqref{g8.11} and \eqref{g8.2}.
\end{lemma}

\begin{proof}
  For every $t\in (0,T]$ we have
    \begin{align*}
      \sup_{\xi\in (0,t]} \Big(V(\xi)^{-1} \int_\xi^t g(s)\dint s \int_0^\xi \varphi(\tau)\dint\tau\Big) & \leq \sup_{\xi\in (0,T]} \Big(V(\xi)^{-1} \int_\xi^Tg(s)\dint s\int_0^\xi \varphi(\tau)\dint\tau\Big) \\
          &= \wt{\varrho}_0(g).
    \end{align*}
    Therefore,
    \[
    \int_\xi^t g(s)\dint s \int_0^\xi \varphi(\tau)\dint\tau\leq \wt{\varrho}_0(g)V(\xi), \quad \xi\in (0,t], \ t\in (0,T].
    \]
    Thus
    \begin{align*}
      \int_0^t \varphi(\xi)\left(\int_\xi^t g(s)\dint s\right)\dint\xi &\leq \int_0^t\left(\int_0^\xi \varphi(\tau)\dint\tau\right)\left(\int_\xi^t g(s)\dint s\right)\xi^{-1} \dint\xi \\
      &\leq \wt{\varrho}_0(g) \int_0^t V(\xi)\xi^{-1}\dint\xi,
     \end{align*}
and, according to \eqref{g8.21},
\[
\sup_{t\in (0,T]} V(t)^{-1} \int_0^t \varphi(\xi) \left(\int_\xi^t g(s)\dint s\right)\dint\xi \leq c  \wt{\varrho}_0(g).
\]
Consequently, by \eqref{g8.11},
\begin{align*}
  \widehat{\varrho}_0(g) = & \sup_{t\in (0,T]} V(t)^{-1} \int_0^t \varphi(\xi)\left(\int_\xi^T g(s)\dint s\right)\dint\xi \\
     \leq & \sup_{t\in (0,T]} V(t)^{-1} \int_0^t \varphi(\xi)\left(\int_\xi^t g(s)\dint s\right)\dint\xi\\
& ~\qquad      +\sup_{t\in (0,T]} V(t)^{-1} \left( \int_0^t \varphi(\xi)\dint\xi\right) \left(\int_t^T g(s)\dint s\right)\\
        \leq & (c+1) \wt{\varrho}_0(g).      
\end{align*}
Since $\wt{\varrho}_0(g) \leq \widehat{\varrho}_0(g)$, we obtain \eqref{g8.22}.
\end{proof}

\begin{corollary}
  Let the assumptions of Lemma~\ref{lemma-g-8.1} be satisfied and the estimates \eqref{g8.10} and \eqref{g8.21} be valid. Then
  \beq
  \varrho_0(g) \simeq \wt{\varrho}_0(g),\quad g\in M^+(0,T),
  \label{g8.23}
  \eeq
with constants not depending on $g$, recall \eqref{g7.11} and \eqref{g8.2}.
  \label{cor-g-8.8}
  \end{corollary}
\begin{proof}
Plainly \eqref{g8.11} and \eqref{g8.22}\ imply \eqref{g8.23}.
  \end{proof}

\begin{lemma}
\label{lemma-g-8.9}
Let the assumptions of Lemma~\ref{lemma-g-8.1} be satisfied, the estimate \eqref{g8.13} be valid, and assume that for some $\varepsilon>0$ the function $U_1(t)t^\varepsilon $ is decreasing on $(0,T]$, recall \eqref{g3.12}-\eqref{g3.13}. Then the estimate \eqref{g8.23} holds for the norms \eqref{g7.11} and \eqref{g8.2}.
\end{lemma}

\begin{proof}
  Without loss of generality we assume that $U_1(T)=1$ and introduce a discretizing sequence $\left\lbrace \nu_m\right\rbrace _{m\in\no}$ by
  \beq
  \nu_m=\sup\left\{t\in (0,T]: U_1(t)=2^m\right\}, \quad m\in\no.
  \label{g8.24}
  \eeq
  Note that $U_1$ is a positive and decreasing function on $(0,T]$, with $U_1(0+)=\infty$, such that $\nu_m$ is well-defined, $m\in\no$, and
  \beq
  \nu_0=T, \quad 0<\nu_{m+1}<\nu_m, \quad m\in\no,\quad \lim_{m\to\infty} \nu_m=0.
  \label{g8.25}
  \eeq
  By assumption, $U_1(t) t^\varepsilon$ decreases which leads to
  \beq
  \nu_{m+1}<\nu_m\leq2^{1/\varepsilon} \nu_{m+1}\quad\text{such that}\quad \nu_{m+1}\simeq\nu_m, \quad m\in\no,
  \label{g8.26}
  \eeq
for fixed $\varepsilon$. For convenience we use the notation
  \beq
\Delta_m=(\nu_{m+1},\nu_m],\quad m\in\no.
  \label{g8.27}
  \eeq
  The discretized version of \eqref{g8.15} then yields
  \begin{align}
    \widehat{\varrho}_1(g) & = \sup_{m\in\no}\sup_{t\in\Delta_m} U_1(t) \int_0^t \xi^{k/n} g(\xi)\dint\xi \nonumber\\
    & \simeq \sup_{m\in\no} 2^m \sup_{t\in\Delta_m} \int_0^t \xi^{k/n} g(\xi)\dint\xi \nonumber\\
    &= \sup_{m\in\no} 2^m\int_0^{\nu_m} \xi^{k/n} g(\xi)\dint\xi \nonumber\\
    &=  \sup_{m\in\no} 2^m \sum_{j\geq m} \int_{\Delta_j} \xi^{k/n} g(\xi)\dint\xi\label{g8.28}.
\end{align}
  Here we used the assertion
  \beq
  U_1(t)\simeq 2^m, \quad t\in\Delta_m,\quad m\in\no.
  \label{g8.29}
  \eeq
  Now we apply some well-known estimate for non-negative sequences $\{\alpha_m\}_{m\in\no}$ and positive sequences $\{\beta_m\}_{m\in\no}$ which satisfy, in addition, $\beta_{m+1}/\beta_m\geq B>1$  for some number $B$. Then for $0<p\leq\infty$ and $1\leq r\leq\infty$,
  \beq
  \left(\sum_{m\in\no} \left(\beta_m \left(\sum_{j\geq m} \alpha_j^r\right)^{\frac1r} \right)^p \right)^{\frac1p} \leq c(B,p)\left(\sum_{m\in\no} \left(\beta_m \alpha_m\right)^p\right)^{\frac1p},
  \label{g8.30}
  \eeq
  where $c(B,p)$ is a positive constant (and the usual modification for $p=\infty$ or $r=\infty$). Since the inequality inverse to \eqref{g8.30} is valid (with $c=1$), the estimate \eqref{g8.30} is in fact an equivalence. Now we use \eqref{g8.30} with
  \[\beta_m=2^m, \quad r=1, \quad p=\infty, \quad \alpha_j=\int_{\Delta_j} \xi^{k/n} g(\xi)\dint\xi, \]
  insert it in \eqref{g8.28} and conclude
  \beq
  \widehat{\varrho}_1(g) \leq \ c_1\ \sup_{m\in\no} \ 2^m \int_{\Delta_m} \xi^{k/n} g(\xi)\dint\xi \leq c_2\ \sup_{m\in\no} \ 2^m \nu_m^{k/n} \int_{\Delta_m} g(\xi)\dint\xi,
\label{g8.31}
  \eeq
  where we used in the latter estimate that $\xi \simeq \nu_m$ for $\xi\in\Delta_m$, $m\in\no$. It remains to estimate $\wt{\varrho}_0(g)$ given by \eqref{g8.2} from below. By similar discretization arguments as above we observe that
  \begin{align*}
    \wt{\varrho}_0(g) & \geq \sup_{m\in\no}\sup_{t\in\Delta_m} V(t)^{-1} t \varphi(t) \int_t^T g(\xi)\dint\xi \\
    & \geq \ \sup_{m\in\nat} \left(\int_{\Delta_{m-1}} g(\xi)\dint\xi\right) \sup_{t\in\Delta_m} V(t)^{-1} t \varphi(t)\\
    &\simeq \sup_{m\in \nat} \nu_{m-1}^{k/n} \int_{\Delta_{m-1}} g(\xi)\dint\xi \sup_{t\in\Delta_m} V(t)^{-1} t^{1-k/n}\varphi(t),
  \end{align*}
  where we used that $\varphi$ decreases and the obvious estimate
  \[
\int_t^T g(\xi)\dint\xi \geq \int_{\nu_m}^T g(\xi)\dint\xi \geq \int_{\Delta_{m-1}} g(\xi)\dint\xi,\quad t\in\Delta_m,\quad m\in\nat.
  \]
    Thus \eqref{g8.24} implies
    \begin{align*}
      2^{m+1} = & \ U_1(\nu_{m+1}) = \max\left\{ U_1(\nu_m), \sup_{t\in\Delta_m} V(t)^{-1} t^{1-k/n}\varphi(t)\right\} \\
      = & \sup_{t\in\Delta_m} V(t)^{-1} t^{1-k/n}\varphi(t).
    \end{align*}
    Finally this leads to
    \[
     \wt{\varrho}_0(g) \geq c_3 \sup_{m\in \nat} \nu_{m-1}^{k/n}2^{m+1}\int_{\Delta_{m-1}} g(\xi)\dint\xi = 4 c_3 \sup_{m\in \no} \nu_{m}^{k/n}2^{m}\int_{\Delta_m} g(\xi)\dint\xi.
    \]
    Together with \eqref{g8.31} this results in
    \[
    \wt{\varrho}_0(g) \geq c_4 \widehat{\varrho}_1(g),\quad g\in M^+(0,T),
    \]
where $c_4>0$ does not depend on $g$. In view of \eqref{g8.14} this yields \eqref{g8.23} as desired.
\end{proof}

\begin{remark}
\label{rem-g-8.10}
In Lemma~\ref{lemma-g-8.9} we considered the alternative situation to Corollary~\ref{cor-g-8.8}, where the estimate \eqref{g8.10} is replaced by \eqref{g8.13}. In this more delicate case a more flexible method of discretization was needed for the proof.
\end{remark}

\subsection{The case $E=\Lambda_q(v)$, $1<q<\infty$}\label{sect-g-9}
Now we deal with Lorentz spaces $\Lambda_q(v)$ for $1<q<\infty$, recall Example~\ref{ex-Lorentz} and \eqref{g3.1}. Our main aim is to obtain some counterparts of Corollary~\ref{cor-g-8.8} and Lemma~\ref{lemma-g-8.9} from the preceding section, but now corresponding to the case $q>1$. For this we need some auxiliary lemmas first.

\begin{lemma}
  \label{lemma-g-9.1}
  Let the assumptions \eqref{g1.3}, \eqref{g1.4}, and \eqref{g3.1}-\eqref{g3.3} be satisfied with $1<q<\infty$. Then the following estimate holds for the norm \eqref{g7.11},
  \beq
  \varrho_0(g)\simeq \wt{\varrho}_0(g) + \varrho_1(g)+\varrho_2(g), \quad g\in M^+(0,T),
  \label{g9.1}
  \eeq
  where
  \begin{align}
  \wt{\varrho}_0(g)& =\left(\int_0^T \left(\left(\int_0^t \varphi(\tau)\dint\tau\right)\left(\int_t^T g(s)\dint s\right)\right)^{q'} w(t)\dint t\right)^{1/q'},
  \label{g9.2}\\
\varrho_1(g)&= \left(\int_0^T \left(\int_0^t \Phi_k(\xi,t)g(\xi) \dint\xi\right)^{q'} w(t)\dint t\right)^{1/q'},
\label{g9.3}\\
\varrho_2(g)&= \left(\int_0^T \Phi_k(\xi,T)g(\xi)\dint\xi\right)\left(\int_T^\infty w(t)\dint t\right)^{1/q'},
\label{g9.4}
  \end{align}
where $\Phi_k$ is defined by \eqref{g8.4} and $w$ by \eqref{g7.2}.
\end{lemma}

\begin{proof}
Note that the line of arguments is similar to those strengthened in the proof of Lemma~\ref{lemma-g-8.1}. Recall that for an RIS $\wt{E}=\Lambda_q(v)$ the associated RIS is the space $\wt{E}'=\Gamma_{q'}(w)$ with the norm
  \[
  \left\|f\right\|_{\wt{E}'(\real_+)} = \left(\int_0^\infty \left(\int_0^t f^\ast(\tau)\dint\tau\right)^{q'} w(t)\dint t\right)^{1/q'}.
  \]
  Since in our case $f(\tau)=\Psi_0(g,\tau)=f^\ast(\tau)$, see \eqref{g8.5}, thus
  \beq
  \varrho_0(g)=\left\|\Psi_0(g)\right\|_{\wt{E}'(\real_+)} = \left(\int_0^\infty \left(\int_0^t \Psi_0(g,\tau)\dint\tau\right)^{q'} w(t)\dint\tau\right)^{1/q'}.
  \label{g9.5}
  \eeq
  Substitution of \eqref{g8.5} into \eqref{g9.5} gives
  \beq
  \varrho_0(g)\simeq \widehat{\varrho}_1(g)+\varrho_2(g),
  \label{g9.6}
  \eeq
  where
  \begin{align}
    \widehat{\varrho}_1(g)&= \left(\int_0^T \left(\int_0^t \Psi_0(g,\tau)\dint\tau\right)^{q'} w(t)\dint t\right)^{1/q'},
\label{g9.7}\\
\varrho_2(g)&= \left(\int_0^T \Psi_0(g,\tau)\dint\tau\right)\left(\int_T^\infty w(t)\dint t\right)^{1/q'}.
    \label{g9.8}
    \end{align}
  Now we introduce some function $G_k$ for $t\in (0,T]$, and $\tau\in (0,t]$,
    \beq
    G_k(t,\tau)=\int_\tau^t g(\xi)\dint\xi + \tau^{-k/n} \int_0^\tau \xi^{k/n} g(\xi)\dint\xi.
    \label{g9.9}
    \eeq
    Then \eqref{g7.10} and \eqref{g8.5} imply that
    \begin{align}
      \int_0^t \Psi_0(g,\tau)\dint\tau  = & \int_0^t G_k(T,\tau)\varphi(\tau)\dint\tau\nonumber\\
      = & \int_0^t \varphi(\tau)\dint\tau \int_t^T g(\xi)\dint\xi + \int_0^t G_k(t,\tau)\varphi(\tau)\dint\tau.\label{g9.10}
    \end{align}
    After some change of the order of integration we obtain
    \beq
    \int_0^t G_k(t,\tau)\varphi(\tau)\dint\tau = \int_0^t \Phi_k(\xi,t)g(\xi)\dint\xi,
    \label{g9.11}
    \eeq
    recall \eqref{g9.9} and \eqref{g8.4}. As a special case we get
    \beq
    \int_0^T \Psi_0(g,\tau)\dint\tau = \int_0^T G_k(T,\tau)\varphi(\tau)\dint\tau = \int_0^T \Phi_k(\xi,T)g(\xi)\dint\xi,
    \label{g9.12}
    \eeq
    which means the coincidence of \eqref{g9.8} and \eqref{g9.4}. Moreover, substituting \eqref{g9.10} into \eqref{g9.7}, we arrive at
    \beq
    \widehat{\varrho}_1(g) \simeq \wt{\varrho}_0(g)+\varrho_1(g),\quad g\in M^+(0,T),
    \label{g9.13}
    \eeq
    where $\wt{\varrho}_0(g)$ is given by \eqref{g9.2} and
    \beq
    \varrho_1(g)= \left(\int_0^T \left(\int_0^t G_k(t,\tau)\varphi(\tau) \dint\tau\right)^{q'} w(t)\dint t\right)^{1/q'}.
    \label{g9.14}
    \eeq
Now in view of \eqref{g9.11}, equality \eqref{g9.3} coincides with \eqref{g9.14}, and \eqref{g9.6} and \eqref{g9.13} imply \eqref{g9.1}. 
\end{proof}

\begin{lemma}
  \label{lemma-g-9.2}
  Let $1<q<\infty$, $v>0$ be a measurable function on $(0,T_0)$, where $T_0\in (0,\infty]$, and $V$ and $w$ be given on $(0,T_0)$ by \eqref{g3.2} and \eqref{g7.2}, respectively. Let $\ve>0$, $\delta\in [0, \ve/q)$, and
    \beq
V(t) t^{-\ve}\quad \text{increasing on}\quad  (0,T_0).    \label{g9.15}
\eeq
Then, for $t\in (0,T_0)$,
\begin{align}
  \left(\int_t^{T_0} w(\tau)\tau^{\delta q'}\dint\tau\right)^{1/q'} & \leq \left(\frac{\ve q}{q'(\ve-\delta q)}\right)^{1/q'} t^\delta V(t)^{-1/q}\label{g9.16},\\
  \left(\int_0^t w(\tau)^{-q/q'}\tau^{-(\delta+1)q}\dint\tau\right)^{1/q} & \leq \frac{t^{-\delta} V(t)^{1/q}}{\ve^{1/q'}(\ve-\delta q)^{1/q}} \ .\label{g9.17}  
\end{align}
\end{lemma}

\begin{proof}
  We start proving \eqref{g9.16}. Here we apply \eqref{g7.2} and \eqref{g9.15} to conclude
  \begin{align*}
    \int_t^{T_0} w(\tau)\tau^{\delta q'}\dint \tau  \leq & \left(t^\ve V(t)^{-1}\right)^{\delta q'/\ve} \int_t^{T_0} V(\tau)^{q'(\delta/\ve-1)} v(\tau)\dint\tau \\
    = &\ t^{\delta q'} V(t)^{-\delta q'/\ve} (q' (\delta/\ve-1)+1)^{-1} \left(V(\tau)^{q'(\delta/\ve-1)+1}\right) \Big|^{T_0}_{\tau=t}\\
    \leq & t^{\delta q'} \left(q'\left(\frac1q-\frac{\delta}{\ve}\right)\right)^{-1} V(t)^{1-q'},
    \end{align*}
  which implies \eqref{g9.16}.

  It remains to verify \eqref{g9.17}. Property \eqref{g9.15} yields that
  \[
  v(\tau)=V'(\tau) \geq \ve \tau^{-1} V(\tau),\quad \tau\in (0,T_0).
  \]
  Therefore,
  \begin{align*}
    \int_0^t w(\tau)^{-q/q'} \tau^{-(\delta+1)q} \dint \tau & = \int_0^t V(\tau)^q v(\tau)^{-q/q'} \tau^{-(\delta+1)q} \dint\tau\\
    & \leq \ve^{-q/q'} \int_0^t V(\tau)\tau^{-\delta q-1} \dint\tau.
  \end{align*}
  Now \eqref{g9.15} implies that
  \[
  \int_0^t V(\tau)\tau^{-\delta q-1} \dint\tau \leq V(t)t^{-\ve} \int_0^t \tau^{\ve-\delta q-1} \dint\tau = V(t) t^{-\delta q} (\ve-\delta q)^{-1}.
  \]
These estimates conclude the proof.  
\end{proof}

We formulate an immediate consequence of the estimates \eqref{g9.16} and \eqref{g9.17}.

\begin{corollary}
  \label{cor-g-9.3}
Let the assumptions of Lemma~\ref{lemma-g-9.2} be satisfied. Then
\beq
\left(\int_0^t w(\tau)^{-q/q'}\tau^{-(\delta+1)q}\dint\tau\right)^{1/q} \left(\int_t^{T_0} w(\tau)\tau^{\delta q'}\dint\tau\right)^{1/q'} \leq \frac{(q/q')^{1/q'}}{\ve-\delta q}.
\label{g9.18}
\eeq
\end{corollary}
  
\begin{lemma}
  \label{lemma-g-9.4}
  Let the assumptions of Lemma~\ref{lemma-g-9.1} be satisfied, and assume the conditions \eqref{g8.10} and \eqref{g9.15} to hold. Then there is the equivalence
  \beq
  \varrho_0(g) \simeq \wt{\varrho}_0(g),\quad g\in M^+(0,T),
  \label{g9.19}
  \eeq
where  the norm $\varrho_0$ is given by \eqref{g7.11} and $\wt{\varrho}_0(g)$ by \eqref{g9.2}.  
\end{lemma}

\begin{proof}
  \underline{\em Step 1}.~   Lemma~\ref{lemma-g-9.1} implies the estimates \eqref{g9.1}--\eqref{g9.4}. Next we apply \eqref{g8.10} to $\Phi_k(\xi,t)$, given by \eqref{g8.4}, that
  \[
  \Phi_k(\xi,t)\simeq \int_0^\xi \varphi(\tau)\dint\tau,\quad t\in (0,T],
    \]
    and thus
 \begin{align*}
   \varrho_1(g) = & \left(\int_0^T \left(\int_0^t \left(\int_0^\xi\varphi(\tau)\dint\tau\right)g(\xi)\dint\xi\right)^{q'} w(t)\dint t\right)^{1/q'}\\
   = & \left(\int_0^T   \left(\int_0^t\varphi(\tau)
\left(\int_\tau^t g(\xi)\dint\xi\right)\dint\tau\right)^{q'} w(t)\dint t\right)^{1/q'}.
 \end{align*}
 Hence,
 \beq
 \varrho_1(g) \leq  c  \left(\int_0^T   \left(\int_0^t\varphi(\tau)
\left(\int_\tau^T g(\xi)\dint\xi\right)\dint\tau\right)^{q'} w(t)\dint t\right)^{1/q'}.
 \label{g9.20}
 \eeq
 Similarly,
 \beq
 \varrho_2(g) \simeq  \left(\int_0^T\varphi(\tau)\left(\int_\tau^T g(\xi)\dint\xi\right)\dint\tau\right)
 \left(\int_T^\infty  w(t)\dint t\right)^{1/q'}.
 \label{g9.21}
 \eeq
 Therefore it is enough to prove that
 \begin{align}
   \varrho_1(g) & \leq \ c_1\ \wt{\varrho}_0(g), \label{g9.22}\\
   \varrho_2(g) & \leq \ c_2 \ \wt{\varrho}_0(g), \label{g9.23}
   \end{align}
 where $c_1,c_2$ are positive constants independent of $g\in M^+(0,T)$. \\

 \underline{\em Step 2}.~ We verify \eqref{g9.22}. We apply Hardy's inequality \cite[Thm.~2, p.~41]{mazya} in adapted notation, that is,
 \begin{align}
 \Big(\int_0^T   &\left(\int_0^t\varphi(\tau)
  \left(\int_\tau^T g(\xi)\dint\xi\right)\dint\tau\right)^{q'} w(t)\dint t\Big)^{1/q'}\nonumber\\
 & \leq c_3 \left(\int_0^T \left(t \varphi(t) \int_t^T g(\xi)\dint\xi\right)^{q'} w(t)\dint t\right)^{1/q'}
   \label{g9.24}
\end{align}
 if, and only if,
 \beq
 B_0=\sup_{t\in (0,T)} \left(\int_t^T w(\tau)\dint\tau\right)^{1/q'} \left(\int_0^t w(\tau)^{-q/q'} \tau^{-q} \dint\tau\right)^{1/q}<\infty.
 \label{g9.25}
 \eeq
 Moreover, for the best possible constant $c_3$ in \eqref{g9.24} we have
 \beq
 B_0\leq c_3\leq B_0\left(\frac{q'}{q'-1}\right)^{\frac{q'-1}{q'}} (q')^\frac{1}{q'} = B_0 q^\frac1q (q')^\frac{1}{q'}.
 \label{g9.26}
 \eeq
 Corollary~\ref{cor-g-9.3} with $\delta=0$, $T_0=T$, implies that
 \beq
 B_0\leq \ve^{-1} (q'-1)^{-1/q'},
 \label{g9.27}
 \eeq
 and consequently
\beq
c_3\leq \frac{q}{\ve}.
\label{g9.28}
 \eeq
Now \eqref{g9.22} follows from \eqref{g9.20}, \eqref{g9.24}, \eqref{g9.2} and from the obvious estimate $t\varphi(t) \leq \int_0^t \varphi(\tau)\dint\tau$ due to the monotonicity of $\varphi$. \\

\underline{\em Step 3}.~ It remains to verify \eqref{g9.23} which is much simpler. By H\"older's inequality we get from \eqref{g9.21} that
\begin{align*}
  \varrho_2(g) \leq & c \left(\int_0^T\tau \varphi(\tau)\left(\int_\tau^T g(\xi)\dint\xi\right)^{q'} w(\tau)\dint\tau\right)^{1/q'} \times   \\
&\quad \times  \left(\int_0^T \tau^{-q} w(\tau)^{-q/q'} \dint\tau\right)^{1/q} 
  \left(\int_T^\infty  w(t)\dint t\right)^{1/q'}\\
  \leq & c\ \wt{\varrho}_0(g) \left(\int_0^T \tau^{-q} w(\tau)^{-q/q'} \dint\tau\right)^{1/q} 
  \left(\int_T^\infty  w(t)\dint t\right)^{1/q'}.
\end{align*}
Now application of \eqref{g9.18} with $T_0=\infty$, $\delta=0$, $t=T$ gives
\[
\varrho_2(g)\leq c_4 \frac1\ve \left(\frac {q}{q'}\right)^{1/q'} \wt{\varrho}_0(g).\]
Thus we have finally shown \eqref{g9.22} and \eqref{g9.23} which together with \eqref{g9.1} imply \eqref{g9.19}.
\end{proof}

The last preparatory lemma we need is the following.

\begin{lemma}
  \label{lemma-g-9.5}
  Let the assumption of Lemma~\ref{lemma-g-9.1} be satisfied. If the estimate \eqref{g8.13} holds, and for the function $U_q$, given by \eqref{g3.13} for $q>1$, there is some $\ve>0$ such that
  \beq
  t^\ve U_q(t)\quad \text{decreases on}\quad (0,T),
  \label{g9.29}
  \eeq
  then the assertion \eqref{g9.19} holds with $\wt{\varrho}_0(g)$ given by \eqref{g9.2}. All the constants appearing in \eqref{g9.19} are positive, finite, and independent of $g\in M^+(0,T)$.
  \end{lemma}

\begin{proof}
  We strengthen a similar line of arguments like in the proof of Lemma~\ref{lemma-g-8.9} and use the method of discretization again.\\

  \underline{\em Step 1.}~ According to \eqref{g8.4} and \eqref{g8.17}-\eqref{g8.18} we have
  \beq
  \Phi_k(\xi,t)\simeq \xi^{k/n} t^{1-k/n}\varphi(t).
  \label{g9.30}
  \eeq
  Hence \eqref{g9.1}-\eqref{g9.4} read as
  \begin{align}
    \varrho_1(g)\simeq & \left(\int_0^T \left(\left(\int_0^t \xi^{k/n} g(\xi)\dint\xi\right) t^{1-k/n} \varphi(t)\right)^{q'} w(t)\dint t\right)^{1/q'}    \label{g9.31}\\
    \varrho_2(g)\simeq & \wt{\varrho}_2(g) = \left(\int_0^T \xi^{k/n} g(\xi)\dint\xi\right) T^{1-k/n} \varphi(T) \left(\int_T^\infty w(t)\dint t\right)^{1/q'}.
    \label{g9.32}
  \end{align}

  \underline{\em Step 2.}~ First we estimate $\varrho_1(g)$. We substitute formulas \eqref{g7.2} and \eqref{g3.12} into \eqref{g9.31} and obtain
  \beq
  \varrho_1(g)\simeq \wt{\varrho}_1(g) = \left(\int_0^T \left(\int_0^t \xi^{k/n} g(\xi)\dint\xi\right)^{q'} \wt{W}(t)^{q'} v(t)\dint t\right)^{1/q'}.
  \label{g9.33}
  \eeq
  Note that $U_q(T)=0$, $U_q(0+)=\infty$, and $U_q(t) t^\ve $ is monotonically decreasing. We introduce the discretizing sequence $\{\delta_m\}_{m\in\ganz}$ by
  \beq
\delta_m=\sup\{\tau\in (0,T): U_q(\tau)=2^m\},\quad m\in\ganz.
  \label{g9.34}
    \eeq
    Thus we observe $\delta_0\in (0,T)$, $\delta_m\to 0$ for $m\to +\infty$, $\delta_m\to T$ for $m\to-\infty$, and
    \beq
\delta_{m+1}<\delta_m\leq \delta_{m+1} 2^{1/\ve},\quad m\in\ganz.
    \label{g9.35}
    \eeq
    We use the notation
    \beq
    \wt{\Delta}_m = (\delta_{m+1},\delta_m], \quad m\in\ganz.
    \label{g9.36}
    \eeq
    Therefore,
    \[
    \wt{\varrho}_1(g)^{q'} = \sum_ {m\in\ganz} \int_{\wt{\Delta}_m} \left(\int_0^t \xi^{k/n} g(\xi)\dint\xi\right)^{q'} \wt{W}(t)^{q'} v(t)\dint t.
    \]
    In view of \eqref{g9.34} this can be continued by
    \begin{align*}
      \wt{\varrho}_1(g)^{q'} \leq &  \sum_ {m\in\ganz}\left( \int_0^{\delta_m}\xi^{k/n} g(\xi)\dint\xi\right)^{q'}
      \int_{\wt{\Delta}_m}  \wt{W}(t)^{q'} v(t)\dint t\\
      = & \sum_ {m\in\ganz}\left( \int_0^{\delta_m}\xi^{k/n} g(\xi)\dint\xi\right)^{q'} \left(U_q(\delta_{m+1})^{q'} - U_q(\delta_m)^{q'}\right)\\
      = & \ (2^{q'}-1) \sum_{m\in\ganz} 2^{mq'} \left(\sum_{j\geq m}\int_{\wt{\Delta}_j }\xi^{k/n} g(\xi)\dint\xi\right)^{q'} .          
    \end{align*}
    Now we apply in appropriately adapted notation estimate \eqref{g8.30} again and obtain
    \[
     \wt{\varrho}_1(g)\leq c(q') \left(\sum_{m\in\ganz} 2^{mq'} \left(\int_{\wt{\Delta}_m }\xi^{k/n} g(\xi)\dint\xi\right)^{q'}\right)^{1/q'}.
     \]
     Using \eqref{g9.35} and \eqref{g9.36} we observe $ \xi \simeq \delta_m$, $\xi\in \wt{\Delta}_m$, $m\in\ganz$, such that finally
     \beq
     \wt{\varrho}_1(g)\leq c_1(q,\ve) \left(\sum_{m\in\ganz} 2^{mq'} \delta_m^{kq'/n} \left(\int_{\wt{\Delta}_m } g(\xi)\dint\xi\right)^{q'}\right)^{1/q'}.
     \label{g9.37}
     \eeq

     \underline{\em Step 3.}~ We deal with $\varrho_2(g)$ and \eqref{g9.32}. We apply \eqref{g9.16} with $\delta=0$ and obtain
     \beq
\wt{\varrho}_2(g)\leq c_2(q)\left(\int_0^T \xi^{k/n} g(\xi)\dint\xi\right) T^{1-k/n} \varphi(T) V(T)^{-1/q}.
     \label{g9.38}
     \eeq
     Moreover,
     \[
     \int_0^{\delta_0} \xi^{k/n}g(\xi)\dint\xi = \sum_{m\in\no} \int_{\wt{\Delta}_m} \xi^{k/n}g(\xi)\dint\xi \simeq \sum_{m\in\no}\delta_m^{k/n}\int_{\wt{\Delta}_m} g(\xi)\dint\xi,
     \]
such that H\"older's inequality leads to
\[
\int_0^{\delta_0} \xi^{k/n}g(\xi)\dint\xi\leq c\left( \sum_{m\in\no} 2^{-mq}\right)^{\frac1q}\left(\sum_{m\in\no}2^{mq'} 
\delta_m^{q'k/n}\left(\int_{\wt{\Delta}_m} g(\xi)\dint\xi\right)^{q'}\right)^{\frac{1}{q'}},
 \]
 such that
 \[
\int_0^T \xi^{k/n}g(\xi)\dint\xi\leq c_3(q,\ve)\left(\left(\sum_{m\in\no}2^{mq'} 
\delta_m^{q'k/n}\left(\int_{\wt{\Delta}_m} g(\xi)\dint\xi\right)^{q'}\right)^{\frac{1}{q'}}+T^{k/n}\int_{\delta_0}^T g(\xi)\dint\xi\right).
 \]
 Together with \eqref{g9.37}, \eqref{g9.38} this leads to
 \beq
 \wt{\varrho}_1(g)+\wt{\varrho}_2(g)\leq c_4(q,\ve,T)\left(\left(\sum_{m\in\no}2^{mq'} 
 \delta_m^{q'k/n}\left(\int_{\wt{\Delta}_m} g(\xi)\dint\xi\right)^{q'}\right)^{\frac{1}{q'}}+\int_{\delta_0}^T g(\xi)\dint\xi\right).
 \label{g9.39}
 \eeq

 \underline{\em Step 4.}~ We estimate $\wt{\varrho}_0(g)$ in \eqref{g9.2} from below. First of all,
 \begin{align*}
   \wt{\varrho}_0(g)\geq & \left(\int_0^{\delta_0}\left(\left(\int_0^t \varphi(\tau)\dint\tau\right)\left(\int_t^T g(\xi)\dint\xi\right)\right)^{q'} w(t)\dint t\right)^{\frac{1}{q'}} \\
   \geq &\left(\int_{\delta_0}^T g(\xi)\dint\xi\right)\left(\int_0^{\delta_0} \left(t\varphi(t)\right)^{q'} w(t)\dint t\right)^{\frac{1}{q'}},
 \end{align*}
 hence
 \beq
 \int_{\delta_0}^T g(\xi)\dint\xi \leq c(\delta_0,q) \wt{\varrho}_0(g),\quad g\in M^+(0,T).
 \label{g9.40}
 \eeq
 Furthermore, according to \eqref{g9.34}-\eqref{g9.35},
 \begin{align}
   \wt{\varrho}_0(g)^{q'} &= \sum_{m\in\ganz} \int_{\wt{\Delta}_m} \left(\left(\int_0^t \varphi(\tau)\dint\tau\right)\left(\int_t^T g(\xi)\dint\xi\right)\right)^{q'} w(t)\dint t\nonumber\\
   &\geq \sum_{m\in\ganz} \left(\int_{\wt{\Delta}_{m-1}} g(\xi)\dint\xi\right)^{q'} \int_{\wt{\Delta}_m} \left(t\varphi(t)\right)^{q'} w(t)\dint t\nonumber\\
   & \simeq \sum_{m\in \ganz} \left(\int_{\wt{\Delta}_{m-1}} g(\xi)\dint\xi\right)^{q'} \delta_{m-1}^{kq'/n} \int_{\wt{\Delta}_m} \wt{W}(t)^{q'} v(t)\dint t.
   \label{g9.41}
      \end{align}
 For $t\in\wt{\Delta}_m$ we have $t\leq\delta_m<\delta_{m-1}<T$, such that $\int_t^T g(\xi)\dint\xi\geq \int_{\wt{\Delta}_{m-1}} g(\xi)\dint\xi$. In addition, we know that $t\simeq t^{1-/k/n} \delta_{m-1}^{k/n}$ and $\int_0^t\varphi(\tau)\dint\tau\geq t\varphi(t)$, recall also notation \eqref{g3.13} and \eqref{g7.2}.

 Now by similar arguments as presented in the proof above following \eqref{g9.36} we obtain
 \[
 \int_{\wt{\Delta}_m} \wt{W}(t)^{q'} v(t)\dint t = (2^{q'}-1) 2^{mq'} = 2^{q'}(2^{q'}-1) 2^{(m-1)q'}.
 \]
 Substituting this into \eqref{g9.41} and an index shift  lead to
 \beq
 \wt{\varrho}_0(g)\geq c(\ve,q)\left(\sum_{m\in\ganz} 2^{mq'} \delta_m^{kq'/n} \left(\int_{\wt{\Delta}_m} g(\xi)\dint\xi\right)^{q'}\right)^{1/q'}.
 \label{g9.42}
 \eeq
 Thus the estimates \eqref{g9.39}, \eqref {g9.40} and \eqref{g9.42} result in
 \[
 \wt{\varrho}_1(g) + \wt{\varrho}_2(g) \leq c_5(\ve,\delta_0,q,T)\wt{\varrho}_0(g),\quad g\in M^+(0,T).
 \]
Then the last estimate, together with \eqref{g9.33}, \eqref{g9.32} and \eqref{g9.1} yields \eqref{g9.19}.
\end{proof}

Recall that our aim is to prove Theorem~\ref{theo-g-3.3} above. For that reason we summarize our preceding results in the following theorem.

\begin{theorem}
\label{theo-g-9.6}
Let the conditions of Theorem,~\ref{theo-g-3.3} be  satisfied. Then the associated GBFS $X_0'=X_0'(0,T)$ to the optimal space $X_0=X_0(0,T)$ is generated by the function norm $\varrho_0$ such that
\beq
\varrho_0(g)\simeq \wt{\varrho}_0(g),\quad g\in M^+(0,T),
\label{g9.43}
\eeq
where for $q=1$,
\beq
\wt{\varrho}_0(g)=\sup_{t\in (0,T]} \left(V(t)^{-1} \left(\int_0^t \varphi(\tau)\dint\tau\right)\left(\int_t^T g(s)\dint s\right)\right),
\label{g9.44}
\eeq
and for $1<q<\infty,$  
\beq
\wt{\varrho}_0(g)=\left(\int_0^T\left(\left(\int_0^t \varphi(\tau)\dint\tau\right)\left(\int_t^T g(s)\dint s\right)\right)^{q'}w(t)\dint t\right)^{1/q'},
\label{g9.45}
\eeq
where $\frac{1}{q}+\frac{1}{q'}=1$, as usual. Here again
\beq
V(t)=\int_0^t v(\tau)\dint\tau,\qquad w(t)=V(t)^{-q'} v(t).
\label{g9.46}
\eeq
\end{theorem}

\begin{proof}
  First we assume that condition  {\bfseries (A)} (see \eqref{g3.8} and \eqref{g3.9}) is satisfied. Then for $q=1$  we can apply Corollary~\ref{cor-g-8.8}, and find that \eqref{g8.10} coincides with \eqref{g3.8}, and \eqref{g3.9} implies \eqref{g8.21} (as in the last estimate before Lemma~\ref{lemma-g-9.4} with $\delta=0$). Thus \eqref{g9.43} is just \eqref{g8.23} in this case. If $1<q<\infty$, we receive \eqref{g9.19} by Lemma~\ref{lemma-g-9.4}.\\
  Secondly we consider the situation when {\bfseries (B)} holds, that is, \eqref{g3.10} and \eqref{g3.11}. For $q=1$ we can apply Lemma~\ref{lemma-g-8.9}, since \eqref{g8.13} coincides with \eqref{g3.10} and \eqref{g3.11} provides the required property of $U_1$. This yields \eqref{g8.23} which coincides with \eqref{g9.43}. Finally, if $1<q<\infty$, then \eqref{g9.43} is a consequence of Lemma~ \ref{lemma-g-9.5}.
\end{proof}

\section{The description of the optimal Calder\'on space, II}
\label{sect-3b}
Recall that we already stated in Section~\ref{sect-3} above one of our main results, Theorem~\ref{theo-g-3.3}. Now we are ready to present its proof.

\begin{proof}[Proof of Theorem~\ref{theo-g-3.3}]
  Theorem~\ref{theo-g-9.6} above shows that (under the given assumptions) the associated norm is optimal. So what is left to verify  are  the explicit representations for the optimal norm $\|\cdot\|_{X_0}$ in \eqref{g3.17} and \eqref{g3.18}, respectively, with \eqref{g3.16}. This norm is associated to the norm $\wt{\varrho}_0$ presented in \eqref{g9.44}, \eqref{g9.45}. We benefit from the paper \cite{BaGo}  and an application of \cite[Thm.~1.2]{BaGo} (in appropriately adapted notation)
concludes the argument.
\end{proof}

The combination of Theorem~\ref{theo-g-2.4} and \ref{theo-g-3.3} now yields the following result.

\begin{theorem}\label{theo-g-3.4}
  Let the assumptions of Theorems~\ref{theo-g-1.1} and \ref{theo-g-3.3} be satisfied. Let $q=1$ and $\Psi_1(0+)=0$ or $1<q<\infty$. Then the optimal Calder\'on space for the embedding \eqref{g2.7} has the following norm
  \beq
  \label{g3.19}
\|u\|_{\Lambda^k(C;X_0)} = \|u\|_C + \left(\int_0^T \left(\frac{\omega_k(u;t^{1/n})}{\Psi_q(t)}\right)^q \frac{\dint \Psi_q(t)}{\Psi_q(t)}\right)^{1/q}.
  \eeq
\end{theorem}

\begin{proof}
  Theorem~\ref{theo-g-3.3} states that the GBFS $X_0=X_0(0,T)$ is optimal for the embedding $K\mapsto X$, where $K$ is the cone described by \eqref{g1.10} with \eqref{g1.12}. Then Corollary~\ref{cor-g-2.5} shows that the corresponding Calder\'on space $\Lambda^k(C;X_0)$ is optimal for the embedding \eqref{g2.7}. Thus
  \beq
  \|u\|_{\Lambda^k(C;X_0)} = \|u\|_C + \left\|\omega_k(u;\tau^{1/n})\right\|_{X_0(0,T)}.
  \label{g10.1}
  \eeq
  We substitute \eqref{g3.18} into \eqref{g10.1} and arrive at
  \begin{align}
    \|u\|_{\Lambda^k(C;X_0)} \simeq & \|u\|_C +
    \Psi_q(T)^{-1} \left\|\omega_k(u;\tau^{1/n})\right\|_{L_\infty(T_1,T)}\nonumber\\
    &+ \left(\int_0^T \left(\frac{\left\|\omega_k(u;\tau^{1/n})\right\|_{L_\infty(0,t)}}{\Psi_q(t)}\right)^q \frac{\dint \Psi_q(t)}{\Psi_q(t)}\right)^{\frac1q}.    
\label{g10.2}
  \end{align}
  
  But $\omega_k(u; \tau^{1/n})$ increases with respect to $\tau$, hence
  \begin{align*}
    \left\|\omega_k(u;\tau^{1/n}\right\|_{L_\infty(T_1,T)} & \leq  \omega_k(u;T^{1/n}) \leq 2^k \|u\|_C, \\
    \left\|\omega_k(u;\tau^{1/n})\right\|_{L_\infty(0,t)} ~ & \leq \omega_k(u;t^{1/n}), \quad t\in (0,T).
    \end{align*}
Together with \eqref{g10.2} this implies \eqref{g3.19}.  
\end{proof}

\begin{remark}\label{rem-g-3.5}
  In the case $q=1$ and $\Psi_1(0+)>0$, the embedding \eqref{g2.7} takes place `on the limit of the smoothness', and we obtain $\Lambda^k(C;X_0)=C(\rn)$. According to the results in \cite{GoHa} in this case there exist functions $u\in H^G_E(\rn)$ such that $\omega_k(u;t^{1/n})\to 0$ for $t\to 0+$ arbitrarily slowly. Note that in this case $X_0(0,T)=L_\infty(0,T)$ by \eqref{g3.17}, such that the above norm
  \[
  \|u\|_{\Lambda^k(C;X_0)} = \|u\|_C + \left\|\omega_k(u;\tau^{1/n})\right\|_{L_\infty(0,T)} \simeq \|u\|_C.
\]
In that case we cannot say anything else than $u\in C(\rn)$.
\end{remark}

\begin{remark}\label{rem-g-3.6}
  We return to Examples~\ref{exm-g-1.3}, \ref{exm-g-1.4}. If $\alpha\neq k$, then Theorems~\ref{theo-g-3.3}, \ref{theo-g-3.4} can be applied. In the limiting case $\alpha=k$ some special care is needed. This follows from estimate \eqref{g1.20} in Example~\ref{exm-g-1.3}. In Example~\ref{exm-g-1.4} we have the equality \eqref{g1.24} when $\lambda$ is slowly varying on $(0,T]$. So in some sense \eqref{g1.12} can be understood as a special case of \eqref{g1.24} with $\lambda\equiv 1$. Thus Remark~\ref{rem-g-3.2} applies and implies that we can use Theorems~\ref{theo-g-3.3}, \ref{theo-g-3.4} in case of $\alpha\neq k$.  
\end{remark}

Before we can state our next main result, Theorem~\ref{theo-g-3.7} below, which also covers the delicate limiting case $\alpha=k$, we need some further preparation.

\begin{lemma}
  \label{lemma-g-10.1}
  Let $\lambda>0$ be a continuous function on $(0,T]$ such that for some $\delta\in (0,1)$ the function $\lambda(t)t^{-\delta}$ decreases. Then
  \[
  \lambda(s) + \int_s^t \lambda(\tau)\tau^{-1}\dint \tau \leq \frac{\lambda(s)}{\delta} \left(\frac{t}{s}\right)^\delta,\quad s\in (0,t], \quad t\in (0,T].
      \]
    \end{lemma}

\begin{proof}
  We use the assumed monotonicity and argue as follows,
  \begin{align*}
    \int_s^t\lambda(\tau)\tau^{-1}\dint\tau \leq \lambda(s)s^{-\delta}\int_s^t \tau^{\delta-1}\dint\tau = \frac{\lambda(s) s^{-\delta}}{\delta}\left(t^\delta-s^\delta\right).
    \end{align*}
  Hence
  \[
  \lambda(s)+\int_s^t\lambda(\tau)\tau^{-1}\dint\tau  \leq \frac{\lambda(s)}{\delta}\left(\frac{t}{s}\right)^\delta.\]
\end{proof}

\begin{corollary}
  Let $\varphi$ be given by \eqref{g3.15} with $\alpha=k$, and $\lambda>0$ be slowly varying on $(0,T]$. Then for every $\delta\in (0,1)$ the function $\Phi_k$ given by \eqref{g8.4} can be estimated by
    \beq
0<\Phi_k(\xi,t)\leq c\ t^\delta \xi^{k/n-\delta}\lambda(\xi),\quad \xi\in (0,t],\quad t\in (0,T],
    \label{g10.3}
    \eeq
    where $c=c(\delta,k,n)>0$.
\label{cor-g-10.2}
\end{corollary}

\begin{proof}
  If $\varphi$  is given by \eqref{g3.15} with $\alpha=k$, then
  \beq
\int_0^\xi \varphi(\tau)\dint\tau = \int_0^\xi \tau^{k/n-1}\lambda(\tau) \dint\tau \simeq \xi^{k/n}\lambda(\xi),
  \label{g10.4}
  \eeq
  recall \eqref{g5.8}. Hence
  \[
  \Phi_k(\xi,t)=\int_0^\xi\varphi(\tau)\dint\tau + \xi^{k/n}\int_\xi^t \tau^{-k/n}\varphi(\tau)\dint\tau \simeq \xi^{k/n}\left(\lambda(\xi)+\int_\xi^t \lambda(\tau)\tau^{-1}\dint\tau\right). 
  \]
Application of Lemma~\ref{lemma-g-10.1} leads to \eqref{g10.3}.
\end{proof}

\begin {corollary}
  Let the assumptions of Corollary~\ref{cor-g-10.2} be satisfied, and let $\varrho_1$ be given by \eqref{g8.3}-\eqref{g8.4}. Then for any $\delta\in(0,1)$ there is some positive constant $c_1=c_1(\delta,k,n)$, such that
  \beq
\varrho_1(g)\leq c_1 \sup_{t\in (0,T]} \left(V(t)^{-1} t^\delta\int_0^t s^{\frac{k}{n}-\delta} \lambda(s)g(s)\dint s\right),\quad g\in M^+(0,T).
  \label{g10.5}
  \eeq
  \label{cor-g-10.3}
\end{corollary}

\begin{proof}
This follows immediately by substituting \eqref{g10.3} into \eqref{g8.4}.
\end{proof}

\begin{corollary}
  Let the assumptions of Corollary~\ref{cor-g-10.2} be satisfied, and let $\varrho_1$ be given by \eqref{g8.3}-\eqref{g8.4}.  Then for any $\delta\in(0, \min\{1,k/n\})$ there is some positive constant $c_2=c_2(\delta,k,n)$, such that
  \beq
  \varrho_1(g)\leq c_2 \sup_{t\in (0,T]} \left(V(t)^{-1} t^\delta\int_0^t \tau^{\frac{k}{n}-\delta-1}\lambda(\tau) \left(\int_\tau^t g(s)\dint s\right)\dint\tau\right),\quad g\in M^+(0,T).
  \label{g10.6}
  \eeq
  \label{cor-g-10.4}
    \end{corollary}

\begin{proof}
  Recall that when $\lambda>0$ is slowly varying, we have for $0<\delta<k/n$ that
  \beq
  s^{\frac{k}{n}-\delta} \lambda(s)\simeq \int_0^s \tau^{\frac{k}{n}-\delta.-1}\lambda(\tau)\dint\tau,\quad s\in (0,T].
  \label{g10.7}
  \eeq
 Substituting this into \eqref{g10.5} and changing the order of integration we arrive at \eqref{g10.6}.
\end{proof}

Now we come to our next essential result.

\begin{theorem}\label{theo-g-3.7}
  Let  $\varphi$ be determined by \eqref{g3.15}. Assume that $\alpha\leq k$ and \eqref{g3.9} is satisfied, or $\alpha>k$ and \eqref{g3.11} is satisfied with
  \[
   \widetilde{W}(t)=V(t)^{-1} t^{\frac{\alpha-k}{n} } \lambda(t).
    \]
    Then the formulas \eqref{g3.16}-\eqref{g3.18}, \eqref{g3.19} hold with $\Psi_q$, given by \eqref{g3.5}, where
    \beq
    \label{g3.20}
W(t)=V(t)^{-1} t^{\alpha/n} \lambda(t),\quad t\in (0,T].
   \eeq
\end{theorem}

\begin{proof}
  Note first that it is left to consider the case $\alpha=k$ only since the other cases are already covered by Theorem~\ref{theo-g-3.3}, in particular using condition {\bfseries (A)} if $\alpha<k$, and condition {\bfseries (B)} if $\alpha>k$. So we may assume in the following that $\alpha=k$.\\
  
  \underline{\em Step 1}.\quad First we deal with the case $q=1$. For $\varphi$ given by \eqref{g3.15} we have \eqref{g10.4}. Thus $\wt{\varrho}_0$ given by \eqref{g8.2} can be estimated by
  \[\wt{\varrho}_0(g)\simeq \sup_{\tau\in (0,T]} \left(V(\tau)^{-1} \tau^{k/n}\lambda(\tau)\left(\int_\tau^T g(s)\dint s\right)\right),\quad g\in M^+(0,T), 
  \]
  such that
  \[
  \tau^{k/n} \lambda(\tau)\int_\tau^t g(s)\dint s \leq c\wt{\varrho}_0(g) V(\tau),\quad \tau\in (0,t],\quad t\in (0,T].
  \]
  In view of this we can continue \eqref{g10.6} here by
  \[
  \varrho_1(g)\leq c_3\sup_{t\in (0,T]} \left( V(t)^{-1} t^\delta\left(\int_0^t \tau^{-\delta-1}V(\tau)\dint\tau\right)\right)\wt{\varrho}_0(g),
  \]
  for any $\delta\in (0,\min\{1,k/n\})$. Recall that $V(\tau)\tau^{-\ve}$ is monotonically increasing by \eqref{g3.9}. Assume that $\delta<\ve$. Then
  \[
  \int_0^t \tau^{-\delta-1}V(\tau)\dint\tau\leq V(t) t^{-\ve} \int_0^t \tau^{\ve-\delta-1}\dint\tau = V(t) t^{-\delta}(\ve-\delta)^{-1}.
  \]
  Consequently,
  \[
  \varrho_1(g)\leq c_4 \wt{\varrho}_0(g),\quad g\in M^+(0,T),
  \]
  with $c_4=c_4(\delta,\ve,k,n)>0$. Together with \eqref{g8.1} this shows that \eqref{g9.19} is valid for the norm  $\varrho_0$ given by \eqref{g7.11} which is associated to the optimal one. In other words, we can apply the results given in \eqref{g3.16}-\eqref{g3.18} as explained in the proof of Theorem~\ref{theo-g-3.3}.\\

    \underline{\em Step 2}.\quad Assume now $1<q<\infty$. First  we show that assertion \eqref{g9.43} is valid  with  $\wt{\varrho}_0(g)$ from \eqref{g9.45} where $\varphi$
    is given by \eqref{g3.15} with   $\alpha=k.$ Applying \eqref{g10.4} implies 
    
    \beq
    \label{g10.8}
    \wt{\varrho}_0(g)\simeq \left(\int_0^T  \left( t^{k/n}\lambda(t)\left(\int_t^T g(s)\dint s\right)\right)^{q'}w(t)\dint t\right)^{1/q'}.
    \eeq
    
    Further, \eqref{g10.3}, \eqref{g9.3}, and \eqref{g9.4}  yield
    
      \[
 {\varrho}_1(g)\leq c_{1}\left(\int_0^T  \left( \int_0^t s^{k/n-\delta}\lambda(s)g(s)\dint s\right)^{q'}t^{\delta q'}w(t)\dint t\right)^{1/q'}, 
   \]
    
     \[
    {\varrho}_2(g)\leq c_{2}\left(\int_0^T   s^{k/n-\delta}\lambda(s)g(s)\dint s\right)T^{\delta } \left( \int_T^{\infty} w(t)\dint t\right)^{1/q'}. 
    \]
    
   Now, substituting \eqref{g10.7} into these formulas and some change of order of integration  lead to
   \beq
   \label{g10.9}
   {\varrho}_1(g)\leq \tilde{c}_{1}\left(\int_0^T  \left( \int_0^t \tau^{k/n-\delta-1}\lambda(\tau)\left(\int_{\tau}^T g(s)\dint s\right)\dint \tau\right)^{q'}t^{\delta q'}w(t)\dint t\right)^{1/q'}, 
   \eeq
   
   \beq
   \label{g10.10}
    {\varrho}_2(g)\leq \tilde{c}_{2}\left(\int_0^T   \tau^{k/n-\delta-1}\lambda(\tau)\left(\int_{\tau}^T g(s)\dint s\right)\dint \tau\right)T^{\delta } \left( \int_T^{\infty} w(t)\dint t\right)^{1/q'},
    \eeq
    in view of the obvious estimate $\int_{\tau}^t g(s)\dint s\leq \int_{\tau}^T g(s)\dint s, \,\tau< t\leq T. $
    
    Now we apply in appropriately adapted notation Hardy's inequality (see \cite[Thm.~2, Ch.~1]{mazya}) to estimate the right-hand side in \eqref{g10.9} and obtain that 
   \begin{align}   
    \left(\int_0^T  \left( \int_0^t \tau^{k/n-\delta-1}\lambda(\tau)\left(\int_{\tau}^T g(s)\dint s\right)\dint \tau\right)^{q'}t^{\delta q'}w(t)\dint t\right)^{1/q'}
    \nonumber\\    
    \leq c_{3}  \left(\int_0^T  \left(  t^{k/n}\lambda(t)\int_{\tau}^T g(s)\dint s\right)^{q'}w(t)\dint t\right)^{1/q'}
    \label{g10.11}
  \end{align}
    holds if, and only if,
    \[
    B_{\delta}:=\sup\limits_{t \in (0, T)} \left(\, \left(\int_t^T   \tau^{\delta q'}w(\tau)\dint \tau\right)^{1/q'}  \left(\int_0^t   \tau^{-(\delta+1)q}w(\tau)^{-q/q'}\dint \tau\right)^{1/q} \right) <\infty.
    \]
    This condition is satisfied in view of Corollary \ref{cor-g-9.3}: if $0\leq\delta < \varepsilon /q,$ then 
    
    \beq
    \label{g10.12}
    B_{\delta}\leq \frac{\left( q/q'\right)^{1/q'}}{ \varepsilon -\delta q}.
    \eeq
     Moreover, for the best possible constant $c_3$ in \eqref{g10.11} we have
    \beq
    B_{\delta} \leq c_3\leq  B_{\delta}q^{1/q}\left({q'}\right)^{1/q'}\leq \frac{q}{ \varepsilon -\delta q}.
    \label{g10.13}
    \eeq
    Estimates~\eqref{g10.9},\eqref{g10.11}, and \eqref{g10.8} imply that
    \beq
    \varrho_{1}(g) \leq c_{4} \wt{\varrho}_0(g),\quad g \in M^{+}(0, T),
    \label{g10.14}
     \eeq
     where $c_{4}>0$ is independent of $g$.
     
     Similarly, by H\"older's inequality we get from \eqref{g10.10} in view of \eqref{g10.8}
     \[
      {\varrho}_2(g)\leq {c}_{5}\wt{\varrho}_0(g)\left(\int_0^T   \tau^{-(\delta+1)q}w(\tau)^{-q/q'}\dint \tau\right)^{1/q}T^{\delta }\left( \int_T^{\infty} w(\tau)\dint \tau\right)^{1/q'}.
     \]
     But 
     \[
     T^{\delta }\left( \int_T^{\infty} w(\tau)\dint \tau\right)^{1/q'}\leq \left( \int_T^{\infty} \tau^{\delta q'}w(\tau)\dint \tau\right)^{1/q'},
     \]
   and estimating \eqref{g9.18} with  $t=T,\, T_0=\infty,$    yields
    \beq
   \varrho_{2}(g) \leq c_{6} \wt{\varrho}_0(g),\quad g \in M^{+}(0, T),
   \label{g10.15}
   \eeq
   where $c_{6}>0$ is independent of $g$.
    The assertions \eqref{g9.1}, \eqref{g10.14}, and \eqref{g10.15} imply \eqref{g9.43} with $\wt{\varrho}_0(g)$ from \eqref{g9.45}. 
    
    Recall that here  $\varphi$ is given by \eqref{g3.15}, where    $\alpha=k,$ such that \eqref{g9.45} coincides with \eqref{g10.8}. 
    
    It remains to describe the optimal norm  $\|\cdot\|_{X_0}$ which is associated to the norm $\wt{\varrho}_0(g)$  in \eqref{g10.8}. This description is given by the formulas \eqref{g3.16} - \eqref{g3.18} where we have to consider $\varphi$ given by \eqref{g3.15} with      $\alpha=k.$ Thus, formula \eqref{g3.5} is valid where for given $\varphi$ the  function   $W$  \eqref{g3.4} has the equivalent form \eqref{g3.20}.

\end{proof}

\section{Some explicit descriptions of the optimal Calder\'on space}
\label{sect-4}
Here we present a more detailed consideration in the case of classical Bessel potentials, see Example~\ref{exm-g-1.3}. Note that in Example~\ref{exm-g-4.2} below we extend some preceding results in \cite{GNO-7}. The results presented here were announced in our paper \cite{GoHa-3}.

We start with a preparatory lemma which we shall need in the arguments below.

\begin{lemma}\label{lemma-g-11.1}
  Let $0<T\leq\infty$, $\beta>0$, $\lambda$ be a positive slowly varying function on $(0,T)$,  and
  \[
  A(t)=\int_t^T \tau^{-\beta-1}\lambda(\tau)\dint\tau,\quad t\in (0,T).
  \]
Then there exists some $\ve>0$ such that $t^\ve A(t)$ is monotonically decreasing.
\end{lemma}

\begin{proof} 
	We have the equality
	\[
	\left[ t^{\varepsilon}A(t)\right]'=t^{\varepsilon-1}\left[ \varepsilon \int_t^T \tau^{-\beta-1}\lambda(\tau)\dint \tau -t^{-\beta}\lambda(t)\right].
	\]
	Applying estimate \eqref{g5.10} yields
	\[
	\left[ t^{\varepsilon}A(t)\right]'\leq t^{\varepsilon-1}\left[ t^{-\beta}\lambda(t)(\varepsilon c_{\beta}-1)\right]<0,
	\]	
	if we choose $\varepsilon \in (0, c_{\beta}^{-1}).$     
\end{proof}

\begin{example}\label{exm-g-4.1}
Let $0<\alpha<n$, $1\leq q<\infty$, $v=1$ in \eqref{g3.1}, such that $E=L_q(\rn)$ in Example~\ref{exm-g-1.3}. 
For $0<\alpha<n$, $q=1$, the space $H^G_E(\rn)$ is not embedded into $C(\rn)$. If $q>1$, the criterion for the embedding into $C(\rn)$ reads as
\[
\eqref{g1.6} \quad\text{is true}\iff \alpha>\frac{n}{q}.
\]
If $\frac{n}{q}<\alpha<\min\left(n,k+\frac{n}{q}\right)$, then the optimal Calder\'on space for the embedding \eqref{g2.7} has the norm
\beq
\label{g4.1}
\|u\|_{\Lambda^k(C;X_0)} = \|u\|_C + \left(\int_0^T \left(\frac{\omega_k(u;t^{1/n})}{t^{\alpha/n-1/q}}\right)^q \frac{\dint t}{t}\right)^{1/q}.
\eeq
In particular, this means that $\Lambda^k(C;X_0)$ coincides with the classical Besov space $B^{\alpha-n/q}_{\infty,q}(\rn)$, cf. \cite{nik,T-F1} for further details on Besov spaces.
\end{example}

\begin{proof}
	
\underline{\em Step 1}.\quad In the case considered here we have equivalence \eqref{g1.20}. Without loss of generality we can assume that
\beq
\label{g11.1}
	0<\alpha<n;\quad \varphi(t)=t^{\alpha/n-1},\quad  t \in (0,T].
	\eeq
	For the basic RIS $E(\rn)=L_q(\rn),\, 1\leq q<  \infty, $    the  criterion of the embedding \eqref{g1.6} has the form \eqref{g1.5} , such that
\beq
\label{g11.2}
\eqref{g1.6} \Longleftrightarrow \tau^{\alpha/n-1} \in L_{q'}(0, T) \Longleftrightarrow \alpha>n/q.
\eeq
	It means that in case of $q=1$ the embedding \eqref{g1.6} is impossible for this Example. \\
	
	\underline{\em Step 2}.\quad Now let $1<q<\infty$,  $n/q <\alpha <n$, and $\alpha \leq k$.  We have here $v(t)=1,\, V(t)=t;$                      condition \eqref{g3.9} is satisfied. So we may apply the corresponding results of Theorem \ref{theo-g-3.7}. According to \eqref{g3.4}, and  \eqref{g3.5} with $\varphi$ given by \eqref{g11.1}, and   $n/q <\alpha <n,$ we get
	\beq
	\label{g11.3}
	W(t)=\frac{n}{\alpha}t^{\alpha/n-1},\quad \Psi_q(t)=ct^{\alpha/n-1},\quad t \in (0, T], 
	\eeq
	where     $c=c(\alpha,n,q)>0.$ Therefore, formula \eqref{g3.19} leads to \eqref{g4.1}.\\
	
		\underline{\em Step 3}.\quad Now we consider the case $1<q<\infty;\, \max (n/q,  k)  <\alpha <\min ( n, k+n/q).$          
	Condition \eqref{g3.10} is satisfied, we may apply  Theorem \ref{theo-g-3.3} (case {\bf (B)}). Here, according to \eqref{g3.12}, and \eqref{g3.13} with $\varphi$ given by \eqref{g11.1}  we get
	\begin{align}
	\widetilde{W}(t) & =t^{\frac{\alpha-k}{n}-1}, \label{g11.4}\\
	U_{q}(t) & =\left( \int_t^T \tau^{(\frac{\alpha-k}{n}-1)q'}
	\dint \tau\right)^{1/q'}.
	\label{g11.5}
	\end{align}
%
%
	 We apply Lemma \ref{lemma-g-11.1} to the function $A(t)=U_q(t)^{q'}$ which corresponds to the special case $\lambda(t)=1;\quad \beta=\left[ 1-\frac{\alpha-k}{n}\right]^{q'}-1.$ 	For   $\alpha<k+n/q$ we have $\beta>0,$   such that in view of Lemma \ref{lemma-g-11.1} there exists some $\ve>0$ with the property that
	\[
        t^{\varepsilon q'}A(t)\quad\text{decreases monotonically} \iff 
        t^{\varepsilon}U_q(t)\quad\text{decreases monotonically}.
	\]	
	
	It shows that \eqref{g3.11} is valid. Therefore, we may apply Theorem \ref{theo-g-3.3} (case {\bf (B)})) here, as well as Theorem  \ref{theo-g-3.4}. We arrive at the descriptions  \eqref{g3.18},  \eqref{g3.19} for $1<q<\infty,$ and $\Psi_q$
	given  by \eqref{g11.3}. It leads to \eqref{g4.1}.
	\end{proof}

\begin{example}\label{exm-g-4.2}
  Let $0<\alpha<n$, $1\leq q<\infty$, $1<p<\infty$, $E=\Lambda_q(v)$, recall \eqref{g3.1}, where $v$ is given by
  \beq
  \label{g4.2}
v(t)=t^{q/p-1} b^q(t),\quad t\in (0,T),
  \eeq
  where $b$ is a positive slowly varying continuous function on $(0,T)$. In other words, $E(\rn)$  is a so-called Lorentz-Karamata space, cf. \cite{GNO-7}. Now we explicate Example~\ref{exm-g-1.3}. Note that in this case
  \beq\label{g4.2a}
\Psi_q(t)= \begin{cases} \sup\limits_{\tau\in (0,t]} \tau^{\frac{\alpha}{n}-\frac1p} b(\tau)^{-1}, & q=1, \\ \left(\int_0^t \tau^{q'\left(\frac{\alpha}{n}-\frac1p\right)} b(\tau)^{-q'} \frac{\dint\tau}{\tau}\right)^{1/q'}, & q>1,\end{cases}
  \eeq
for $t\in (0,T)$.
  
  \bli
\item[{\upshape\bfseries (i)}]
If $0<\alpha<\frac{n}{p}$, then $H^G_E(\rn)$ is not embedded into $C(\rn)$.
\item[{\upshape\bfseries (ii)}]
  If $\alpha=\frac{n}{p}$, then we have to distinguish further between $q=1$ and $q>1$: in case of $q=1$, we require also $\Psi_1(0+)=0$, and
  \[
  \eqref{g1.6} \quad\text{holds}\iff \Psi_1(t)=\sup_{\tau\in (0,t]} b(\tau)^{-1}<\infty, \quad t\in (0,T].
  \]
  In case of $q>1$, we arrive at
  \[
  \eqref{g1.6}\quad\text{holds}\iff \Psi_q(t)=\left(\int_0^t b(\tau)^{-q'} \frac{\dint\tau}{\tau}\right)^{1/q'}<\infty,  \quad t\in (0,T].
  \]
  In that case the optimal Calder\'on space for the embedding \eqref{g2.7} has the norm \eqref{g3.16}, where in case of $q>1$ we have
  \[
  \frac{\dint \Psi_q(t)}{\Psi_q(t)} \simeq \frac{b(t)^{-q'}}{\int_0^t b(\tau)^{-q'} \frac{\dint \tau}{\tau}} \ \frac{\dint t}{t}.
  \]\
\item[{\upshape\bfseries (iii)}]
  In case of $\frac{n}{p}<\alpha<\min\left(n,k+\frac{n}{p}\right)$, the optimal Calder\'on space for the embedding \eqref{g2.7} has the norm \eqref{g3.16}, where we conclude from \eqref{g4.2a} that
  \[
  \Psi_q(t) \simeq \begin{cases} t^{\frac{\alpha}{n}-\frac1p} b(t)^{-1}, & q=1, \\ t^{\frac{\alpha}{n}-\frac1p} b(t)^{-1}, & q>1,\end{cases}
  \]
  hence $\frac{\dint\Psi_q(t)}{\Psi_q(t)} \simeq \frac{\dint t}{t}$. 
  \eli
\end{example}

\begin{proof}
	\underline{\em Step 1}.\quad  Here we consider the case   of $\varphi$  given by \eqref{g11.1}. For the basic Lorentz-Karamata space $E=\Lambda_{q}(v)$       with $1\leq q <\infty,\, 1<p<\infty$, and $v(t)$ is given by \eqref{g4.2}.  Applying \eqref{g5.9} in appropriately adapted notation leads to
	\beq
	\label{g11.6}
	V(t)=\int_0^t v(\tau) \dint \tau \simeq t^{q/p}b(t)^q, \quad t \in (0, T),
	\eeq
   such that in view of \eqref{g3.4}
   \[
   W(t) \simeq t^{\frac{\alpha}{n}-\frac{q}{p}}b(t)^{-q}, \quad t \in (0, T).
   \]
	Thus, for  $q=1$ we have  by \eqref{g3.5} that
	\beq
	\label{g11.7}
   \Psi_1(t) \simeq \begin{cases} t^{\frac{\alpha}{n}-\frac1p} b(t)^{-1}, & \alpha >n/p,\\ B_{1}(t), & \alpha =n/p,\\
   	\infty,  & \alpha< n/p,
   \end{cases}
		\eeq
	where
	\beq
	\label{g11.8}
	B_{1}(t)=\sup \left\lbrace b(\tau)^{-1}: \tau \in (0, t]\right\rbrace.
	\eeq
	For   $1<q<\infty$   we get by \eqref{g3.5}
	\beq
	\label{g11.9}
	\Psi_q(t) \simeq \left( \int_0^t \tau^{q'(\frac{\alpha}{n}-\frac1p)} b(\tau)^{-q'}\tau^{-1}\dint \tau\right)^{1/q'},
	\eeq
	such that in view of \eqref{g5.9}, for $t \in (0, T)$,
	\beq
	\label{g11.10}
	\Psi_q(t) \simeq \begin{cases} t^{\frac{\alpha}{n}-\frac1p} b(t)^{-1},& \alpha>n/p,\\ 
	B_{q}(t), & \alpha =n/p,\\
	\infty,   & \alpha< n/p,
	\end{cases}
	\eeq
	
	where
	\beq
	\label{g11.11}
	B_{q}(t)=\left\lbrace \int_0^t b(\tau)^{-q'}\tau^{-1}\dint \tau \right\rbrace^{1/q'}.
	\eeq
	According to Lemma \ref{lemma-g-3.1},
	\[
	\eqref{g1.6}\Longleftrightarrow 	\eqref{g1.5}\Longleftrightarrow \Psi_q(T)<\infty.
	\] 
	
	We see that \eqref{g1.6} is impossible for  $\alpha<n/p;$ \eqref{g1.6} is valid for $\alpha>n/p;$ and the validity of \eqref{g1.6} for $\alpha=n/p$      depends on  $B_{q}:\, \eqref{g1.6} \Longleftrightarrow B_{q}(T)<\infty,$  where   $B_{q}$  is given by \eqref{g11.8}, \eqref{g11.10}.\\
	
	\underline{\em Step 2}.\quad  Assume that the conditions of embedding \eqref{g1.6} are satisfied, in particular, for $\alpha\geq n/p.$ If $\alpha\leq k$
	we may apply Theorem \ref{theo-g-3.7} with   $\lambda(t)=1.$ Condition \eqref{g3.9} is valid for function $V$ given by \eqref{g11.6}  for any $\varepsilon \in (0, q/p).$ Indeed, in this case  $V(t)t^{-\varepsilon}=t^{q/p-\varepsilon}b(t)^q$ increases monotonically,          because $b(t)^q$        is slowly varying. Then, by Theorem \ref{theo-g-3.7} we get the descriptions \eqref{g3.16}-\eqref{g3.18} with  $\Psi_q$  from \eqref{g11.7} in case of $q=1$,  or from \eqref{g11.10} in case of $1<q<\infty$.\\
	
	\underline{\em Step 3}.\quad  Finally, let	 $\alpha\geq n/p$, $k< \alpha< \min ( n, k+n/p)$.  We need to verify  condition \eqref{g3.11}. According to \eqref{g3.12} with  $\varphi$ given by  \eqref{g11.1}, and $V$ given by \eqref{g11.6}   we obtain
	\[
\widetilde{W}(t)=t^{\frac{\alpha-k}{n}-\frac{q}{p}}b(t)^{-q},\quad t \in (0,T).
	\]
	
	Here,   $\frac{\alpha-k}{n}-\frac{1}{p}<0,$   such that for    $q=1$ the function $ \widetilde{W}(t)$ decreases monotonically, and
	\beq
	\label{g11.12}
	U_1(t)=\widetilde{W}(t)=t^{\frac{\alpha-k}{n}-\frac{1}{p}}b(t)^{-1},
	\eeq
	see \eqref{g3.13}. For  $1<q<\infty$    we obtain by \eqref{g3.13}  and \eqref{g4.2} that
		\beq
	\label{g11.13}
	U_q(t)=\left( \int_t^T \tau^{q'\left( \frac{\alpha-k}{n}-\frac{1}{p}\right)}b(\tau)^{-q'}\tau^{-1}\dint \tau\right)^{1/q'}.
	\eeq
	We see that for $\alpha<k+n/p$ condition \eqref{g3.11} is valid. For  $U_{1}$ given by  \eqref{g11.12} this is obvious because  $b(t)^{-1}$    is slowly varying. For  $U_{q}$ given by \eqref{g11.13} we define 
	\[
	\beta=\left[\frac{1}{p}-\frac{\alpha-k}{n}\right] q'>0;\quad \lambda(t)=b(t)^{-q'};\quad  A(t)=U_{q}(t)^{q'}.
	\]
	Then, by Lemma \ref{lemma-g-11.1} there exists some $\ve>0$ such that
	\beq
	\label{g11.14}
A(t)t^{\varepsilon q'}\quad\text{decreases monotonically} \iff U_{q}(t)t^{\varepsilon}\quad\text{decreases monotonically}.
	\eeq
	Consequently, we may apply Theorem \ref{theo-g-3.3}  (case {\bf (B)}) and get descriptions \eqref{g3.16}-\eqref{g3.18}, \eqref{g3.19}  with   $\Psi_q$   determined by \eqref{g11.7} for  $q=1,$ or by \eqref{g11.10} for    $1<q<\infty.$ 
	
\end{proof}

\begin{remark}
  \label{rem-g-11.2}
  Note that for   $1<q<\infty$ and the function   $\Psi_q$    defined by integral \eqref{g11.9} we have in view of estimate  \eqref{g5.9} that 
 \beq
 \label{g11.15}
 \frac{\dint \Psi_q(t)}{\Psi_q(t)}\simeq
  \frac{\dint \left( \Psi_q^{q'}(t)\right) }{\left( \Psi_q^{q'}(t)\right)}\simeq \frac{\dint t}{t}.
 \eeq
  For  the function  $\Psi_q=B_q$  given by integral \eqref{g11.11} we get
\beq
\label{g11.16}
\frac{\dint \Psi_q(t)}{\Psi_q(t)}\simeq
\frac{ b(t)^{-q'}t^{-1}\dint t}{\int_0^t b(\tau)^{-q'}\tau^{-1}\dint \tau}.
\eeq
  These assertions are useful when we apply formulas \eqref{g3.16}, \eqref{g3.19}.
  
  \end{remark}

\section*{Acknowledgements}
The work of Elza Bakhtigareeva and Mikhail L. Goldman is supported by the Russian Science Foundation under grant no.19-11-00087 and performed in the Steklov Mathematical Institute of the Russian Academy of Sciences.

\end{document}